\documentclass[11pt]{article}
\usepackage{titlesec}

\titleformat{\paragraph}
{\normalfont\normalsize\bfseries}{\theparagraph}{1em}{}
\titlespacing*{\paragraph}
{0pt}{3.25ex plus 1ex minus .2ex}{1.5ex plus .2ex}
\usepackage{amsfonts}
\usepackage{amsthm}
\usepackage{changes}
\usepackage{amssymb}
\usepackage{latexsym}
\usepackage{amsmath,amsfonts}
\usepackage{mathrsfs}
\usepackage{cases}
\usepackage{latexsym,bm}
\usepackage{indentfirst}
\usepackage{color}
\usepackage{ifpdf}
\usepackage{graphicx}
\usepackage{psfrag}
\usepackage{dsfont}
\usepackage{graphicx}
\usepackage{epstopdf}
\usepackage[
pdfauthor={CSY},
pdftitle={Structural properties of edge-chromatic critical multigraphs},
pdfstartview=XYZ,
bookmarks=true,
colorlinks=true,
linkcolor=blue,
urlcolor=blue,
citecolor=blue,
bookmarks=true,
linktocpage=true,
hyperindex=true
]{hyperref}

\newcommand{\xiaowuhao}{\fontsize{9pt}{\baselineskip}\selectfont}

\newtheorem{THM}{\textbf{Theorem}}[section]
\newtheorem{DEF}{\textbf{Definition}}
\newtheorem{LEM}[THM]{Lemma}
\newtheorem{CON}{\textbf{Conjecture}}
\newtheorem{COR}[THM]{Corollary}

\newtheorem{REM}{\textbf{Remark}}

\newtheorem{CLA}{\textbf{Claim}}[section]
\newtheorem{case}{Case}
\newtheorem{subcase}{Case}
\numberwithin{subcase}{case}
\newtheorem{subsubcase}{Case}
\numberwithin{subsubcase}{subcase}

\newcommand{\CC}{\mathcal{C}}

\newcommand{\D}{\Delta}
\newcommand{\phibar}{\overline{\varphi}}
\newcommand{\phiv}{\varphi}
\begin{document}

\title{Structural properties of edge-chromatic critical multigraphs}

\author{Guantao Chen and  Guangming Jing\\
{\xiaowuhao Department of Mathematics and Statistics, Georgia State University, Atlanta, GA\,30303 }\\
}

\date{}

\maketitle

\begin{abstract}

Let $G$  be a graph with possible multiple edges but no loops. The density of $G$, denoted by $\rho(G)$, is defined as $\max_{H\subset G,|V(H)|\geq 2}$ $ \lceil \frac{|E(H)|}{\lfloor |V(H)|/2\rfloor}\rceil$. Goldberg\,(1973) and Seymour\,(1974) independently conjectured that if the chromatic index $\chi'(G)$ satisfies $\chi'(G)\ge \Delta(G)+2$ then $\chi'(G) = \rho(G)$, which is commonly regarded as Goldberg's conjecture.  An equivalent conjecture, usually credited to Jakobsen, states that for any odd integer $m \ge 3$, if $\chi'(G) \ge \frac{m\D(G)}{m-1} + \frac{m-3}{m-1}$ then $\chi' (G)= \rho(G)$.  Tashkinov tree technique, a common generalization of Vizing fans and Kierstead paths for multigraphs, has emerged as the main tool to attack these two conjectures. On the other hand, Asplund and McDonald recently showed that there is a limitation of this method. In this paper, we will go beyond Tashkinov Tree and provide a much larger extended structure, by which we see hope to tackle   the conjecture. Applying this new technique,  we show that the Goldberg's conjecture  holds for graphs with $\D(G) \le 39$ or $|V(G)| \le 39$ and the Jakobsen Conjecture holds for $m \le 39$, where the previously known best bound is $23$.  We also improve a number of other related results.
\end{abstract}

\emph{\indent \textbf{Keywords}.} Edge chromatic index; graph density; critical chromatic graph; Tashkinov tree; Extended Tashkinov tree


\section{Introduction}
Graphs in this paper may contain multiple edges but no loops. We will generally follow the notation and terminology defined by Stiebits et al. in~\cite{StiebSTF-Book}. 
Let $G$ be a graph with vertex set $V(G)$ and edge
set $E(G)$.  Denote by $\D(G)$ and $\mu(G)$ the maximum degree and the multiplicity of $G$, respectively. 
When $G$ is clear, we simply denote $\D(G)$ and $\mu(G)$ by $\D$ and $\mu$ respectively for convenience. 
A {\bf $k$-edge-coloring} of a graph $G$ is a map $\phiv$:   $E(G) \longrightarrow \{1, 2, \dots, k\}$ that assigns to every edge $e$ of $G$ a color $\phiv(e) \in \{1, 2, \dots, k\}$ such that no two adjacent edges of $G$ receive the same color.  Denote by  $\CC^k(G)$  the set of all $k$-edge-colorings of $G$.  The {\bf  chromatic index $\chi' :=\chi'(G)$} is the least  integer $k\ge 0$ such that   $\CC^k(G) \ne \emptyset$.     Clearly,  $\chi' \ge \Delta$.  Conversely, Vizing~\cite{Vizing64} showed that $\chi' \le \D + \mu$. The gap between $\D$ and $\D + \mu$ may be large since $\mu$ is unbounded.  To compute the exact value of $\chi'(G)$,
people consider the density $\rho:=\rho(G)$ of $G$ defined below.
\[
\rho(G) := \max_{H\subseteq G,|V(H)|\geq 2|} \lceil \frac{|E(H)|}{\lfloor |V(H)|/2\rfloor}\rceil
\]
Since the edges of $G$ with the same color form a matching, we have $|E(H)|\leq \chi'(G)\lfloor|V(H)|/2\rfloor$ for any $H\subseteq G$. Thus $\chi'(G)\geq\rho(G)$.
A graph  is $G$ called {\bf  elementary} if $\chi' (G)= \rho(G)$. 
 Goldberg\,(1973)~\cite{Goldberg}  and Seymour\,(1974)~\cite{Seymour} independently made the following conjecture, which is commonly referred as Goldberg's conjecture.

\begin{CON}\label{con:GAS}
	If $G$ is a graph with  $\chi' \ge \Delta+2$,  then  $G$ is elementary.
\end{CON}
As mentioned in \cite{ChenGKPS16+}, Goldberg's conjecture is equivalent to saying that if $\chi' \ge \D +2$, then it is the ceiling of the {\it fractional chromatic index} of $G$, which can be computed in polynomial time.  Consequently, the NP-completeness of determining $\chi'$ lies in deciding whether $\chi' = \D, \, \D +1, $ or $\ge \D +2$.  Hence, Goldberg's conjecture
is interesting from a computational complexity standpoint.   This conjecture and topics surrounding it  are  featured in  the book~\cite{StiebSTF-Book} of Stiebitz, Scheide, Toft and Favrholdt and the elegant survey~\cite{McDonaldSurvey15} of McDonald.

A graph  $G$ is called {\bf $k$-critical} if
$\chi'(G) = k+1$ and $\chi'(H) \le k$ for every proper subgraph $H$ of $G$. We also call a graph {\it critical} if it is $k$-critical for some $k\ge \D$. 
Jakobsen in~\cite{Jakobsen73} made the following weaker  conjecture.
\begin{CON}\label{con:Jm1}
	Let $G$ be a critical graph. If $\chi' > \frac{m}{m-1}\D + \frac{m-3}{m-1}$ for an odd integer $m \ge 3$, then
	$|V(G)| \le m-2$.
\end{CON}
Historically, the following conjecture, named the Jakobsen Conjecture, has been investigated intensively in the past. Clearly, the conjecture is equivalent to Goldberg's conjecture and  provides a ``scaler'' for proving Goldberg's conjecture.
\begin{CON}\label{con:Jm}
	If $G$ is a graph with
	$\chi'>\frac{m}{m-1}\Delta+\frac{m-3}{m-1}$ for an odd integer $m \ge 3$,  then $G$ is elementary.
\end{CON}
Clearly, if the Jakobsen Conjecture holds for an odd integer $m$ then it holds for every odd integer $m'$ with $m' \le m$. The Jakobsen Conjecture has been confirmed slowly for $m\le 23$ by a series of papers over the last 40 years: $m=5$ independently by Andersen~\cite{Andersen} (1977), Goldberg~\cite{Goldberg} (1973),  and S{\o}rensen\,(unpublished, page 158 in \cite{StiebSTF-Book});  $m=7$ independently by Andersen~\cite{Andersen} (1977)
and S{\o}rensen (unpublished, page 158 in \cite{StiebSTF-Book});
$m=9$ by  Goldberg~\cite{Goldberg-1984} (1984); $m=11$ independently  by Nishizeki and Kashiwagi~\cite{Nishizeki-Kashiwagi-1990} (1990) and by Tashkinov~\cite{Tashkinov-2000} (2000); $m=13$  by Favrholdt, Stiebitz and Toft~\cite{FavrST2006} (2006) ;  $m=15$ by   Scheide~\cite{Scheide-2010} (2010); and $m=23$ by 
Chen et al.~\cite{ChenGKPS16+}. 
Applying our technique result, we show in this paper that the Jakobsen Conjecture is true up to $m =39$. 

A {\bf $k$-triple} $(G, e, \phiv)$ consists of a $k$-critical graph $G$ with $k \ge \D +1$, an edge $e\in E(G)$ and a coloring $\phiv \in \CC^k(G-e)$. Note that in the above definition we require $k \ge \D +1$, so $\chi' =k+1\geq\D+2$.
Let $(G, e, \phiv)$ be  a $k$-triple. 
For a vertex $v\in V(G)$, denote by  $\phiv(v)$  and $\phibar(v)$ the sets of colors assigned and not-assigned  to edges incident $v$, respectively.  Colors in $\phiv(v)$ and $\phibar(v)$ are called {\it seen} and {\it missing} at $v$, respectively.  For each color $\alpha$, let $E_{\alpha}=\{e\in E(G):\phiv(e) = \alpha\}$.  Clearly,  $E_{\alpha}$ is a matching of $G$ and $G[E_{\alpha}\cup E_{\beta}]$ is a union of disjoint paths or even cycles  with edges alternatively colored with $\alpha$ and $\beta$, named {\it $(\alpha, \beta)$-chains}. 
If we interchange the colors $\alpha$ and $\beta$ on an $(\alpha, \beta)$-chain $C$, then we obtain a new $k$-edge-coloring $\phiv^*$ of $G$. In this case, we say the coloring $\phiv^*$ is obtained from $\phiv$ by {\it recoloring} $C$, and we denote $\phiv^*= \phiv/C$. This operation is called a {\it Kempe change}.  In this paper, our recoloring techniques are based on Kempe changes. 
An $(\alpha,\beta)$-chain is also called an {\it $(\alpha,\beta)$-path} if it is indeed a path.  For each vertex $v\in V(G)$ with $\alpha\in\phibar(v)$ or $\beta\in\phibar(v)$, denote by $P_v(\alpha, \beta, \phiv)$ the unique $(\alpha, \beta)$-path containing $v$. 
Starting from the vertex $v$, we also notice that path $P_v(\alpha, \beta, \phiv)$ naturally   generates a linear order $ \preceq_{P_v(\alpha, \beta, \phiv)}$ for all vertices on the path, i.e.,  $x  \preceq_{P_v(\alpha, \beta, \phiv)} y$ if and only if $x$ is between $v$ and $y$ in $P_u(\alpha, \beta, \phiv)$.
 For any subgraph $H$ of $G$, let 
$\phiv_e(H)  :=   \phiv(E(H)) =  \{\phiv(f)  \ : \ f\in E(H-e)\}$ 
and call each color in $\phiv_e(H)$ an $H$-used color; let $\phibar_v(H) = \cup_{v\in V(H)} \phibar(v)$ and call each color in $\phibar_v(H)$ an {\it $H$-missing} color. Edges with exactly one end-vertex in $V(H)$ are called {\it boundary } edges of $H$.  Denote by $\partial (H)$ the set of boundary edges of $H$. Denote $\partial_{\alpha,\varphi}(H)=\{f:f\in\partial(H),\varphi(f)=\alpha\}$. If $\varphi$ is understandable, we sometimes drop the coloring $\varphi$ and denote $\partial_{\alpha}(H)=\{f:f\in\partial(H),\varphi(f)=\alpha\}$.
\begin{itemize}
	\item  We call  a vertex set $X \subseteq V(G)$ {\it elementary} if $\phibar(v) \cap \phibar(w) = \emptyset$ for any two distinct vertices $v, w\in X$.
	\item  We call a subgraph $H$ {\bf closed} if $\phiv_e(\partial(H)) \cap \phibar_v(H) = \emptyset$, i.e., no color of a boundary edge is $H$-missing.  Moreover, we call $H$ {\bf  strongly closed} if $H$ is closed and all edges in $\partial(H)$ are colored differently.
\end{itemize}

The above two concepts have played important roles in recent development of graph edge chromatic theory. Goldberg's conjecture is equivalent to saying that for every $k$-triple $(G, e, \phiv)$, $V(G)$ is elementary  or there exists a subgraph $H$ of $G$ with $e\in E(H)$ such that $H$ is strongly closed and $V(H)$ is  elementary. Note that if $V(G)$ is elementary, then it is easy to see $\rho(G)=\chi'(G)$. Hence it creates no confusion that we call a subgraph $H$ of $G$ {\bf elementary} if $V(H)$ is elementary. Starting with Vizing's classic result~\cite{Vizing64}, searching for large elementary subgraphs has a long history in the study of graph edge chromatic theory.  Tashkinov~\cite{Tashkinov-2000} developed a method to find some special elementary and closed trees in a $k$-triple. Such trees are called Tashkinov trees.
There are a number of results~\cite{ChenGKPS16+,CYZ-2011,Scheide-2010,FavrST2006} extending Tashkinov trees to larger elementary trees and there are a number of results~\cite{HaxellK2015,McDonald11,Scheide-2010,FavrST2006} discovering some structural properties from the closed property of maximal Tashkinov trees. However, to the best of our knowledge there are no results extending Tashkinov trees to larger trees inheriting both elementary and closed properties.  Given a $k$-triple $(G, e, \phiv)$ and a closed elementary subgraph $H\subseteq V(G)$, we basically show that under some minor conditions, if there exists  a vertex $x\notin X$ such that $V(H)\cup\{x\}$ is elementary, then there exists a closed elementary subgraph $H'$ with $V(H')\supseteq V(H)\cup\{x\}$.  
Applying our results, we improve almost all known results in this area. Our main result will be stated in Section~\ref{bTashsection} after giving formal definitions of Tashkinov trees and their extensions with some properties. In Section~\ref{bTashsection}, we will also show some applications of our results; and we will give the proof of our main result in Section~\ref{bmainproof} due to its length. Section~\ref{4} gives a proof of a basic application of our main Theorem. The proof is very long, but it contains some techniques and ideas which may shed some lights in attacking Goldberg's conjecture.

\section{Tashkinov trees and their extensions}\label{bTashsection}

Let $G$ be a graph and  $e\in E(G)$. A {\bf tree-sequence} $T$ is an alternating sequence $(y_0, e_1, y_1,  e_2,  \cdots,y_{p-1}, e_p, y_p)$ of distinct vertices $y_i$ and edges $e_i$ of $G$  such that  $e_1=e$ and the endvertices of each $e_i$ are $y_{i}$ and $y_r$ for some $r\in \{1, 2, \dots, i-1\}$. Clearly, the edge set of a tree sequence $T$ indeed  induces a tree; and following  the sequence,  all vertices and edges in  $T$ form a linear order $\prec_{\ell}$.  For every element $x\in T$, let $T_x$ be the sequence generated by $x$ and elements $\prec_{\ell} x$, and call it an $x$-{\it segment}. Note that here $x$ could be an edge or a vertex. Denote by $|T|$ the number of vertices in $T$, i.e., $|T|=p+1$ from the above definition. For each edge $f\in\partial(T)$, denote by $a(T,f)$ and $b(T,f)$ the endvertices of $f$ in $T$ and not in $T$, and name them the {\it in-end} and
{\it out-end} of $f$, respectively. If $T$ is understandable, we simply use $a(f)$ and $b(f)$ for convenience. 

 Let $\phiv$ be a $k$-edge-coloring of $G-e$.  For a color $\alpha$, denote by $v(\alpha, T)$ the first vertex missing color $\alpha$ along $\prec_{\ell}$ of $T$ if $\alpha\in\phibar_v(T)$ and the last vertex of $T$ if $\alpha\notin\phibar_v(T)$. If $T$ is clear, we may simply denote $v(\alpha,T)$ by $v(\alpha)$. We sometimes denote $T_{v(\alpha)}$ by $T(v(\alpha))$.

A {\bf Tashkinov tree} of a $k$-triple $(G, e, \phiv)$ is a tree-sequence $T=(y_0, e_1, y_1, \allowbreak e_2,  \cdots,y_{p-1}, e_p, y_p)$  such that  for each $j \ge 1$, $\phiv(e_j)\in \phibar(y_i)$  for some $i<j$.  A Tashkinov tree $T$ is {\it maximal} if there is no Tashkinov tree $T^*$ of the same $k$-triple containing $T$ as a proper subtree. Clearly, all maximal Tashkinov trees are closed.    A Tashkinov tree is called {\it maximum} if $|T|$ is maximum over all $k$-triples with respect to the same graph $G$. 

\begin{THM}\label{bTHM:TashOrigi} {\em [Tashkinov~\cite{Tashkinov-2000}]}
The vertex set of any Tashkinov tree of a $k$-triple $(G, e, \phiv)$  is elementary.
\end{THM}

Let $(G,e,\phiv)$ be a $k$-triple and $H$ be a closed subgraph of $G$.  A color $\delta$ is called a {\it defective color} of $H$ if
$|\partial_{\delta}(H)| > 1$.  Since $H$ is closed, we have $\delta\notin\phibar_v(H)$ in this case.     An edge $f\in \partial(H)$ is called a {\bf connecting edge} of $H$  if $\delta:=\phiv(f)$ is a defective color of $H$ and  there exists a color $\gamma\in  \phibar_v(H) - \phiv_e(H)$  such that $f\in P_{v(\gamma)}(\delta, \gamma, \phiv)$ and
 $f$ is the first edge of $\partial(H)$ along $P_{v(\gamma)}(\delta, \gamma, \phiv)$ starting at $v(\gamma)$.   In this case, we call $\delta$ a {\it connecting} color and $\gamma$ the {\it companion} color of $\delta$. Note that $\gamma\in  \phibar_v(H) - \phiv_e(H)$ means that color $\gamma$ is missing at a vertex in $H$ and is not assigned to any edge of $H$.  

\begin{DEF}\label{bDEF:ETT}
An Extended Tashkinov Tree {\em ({\bf ETT})} of a $k$-triple 
$(G, e, \phiv)$ is a tree-sequence 
$T=(y_0,e_1,y_1,e_2,...,y_{p-1},e_p,y_p)$
such that for each  $e_i$ with $i \ge 2$,  either $\varphi(e_i)\in\phibar_v(T_{y_{i-1}})$  or $T_{y_{i-1}}$ is closed and $e_i$ is a connecting edge of $T_{y_{i-1}}$.
\end{DEF}
Note that in the above definition, the condition imposed on $e_i$ only involves edges incident to $V(T_{y_{i-1}})$. So, if a coloring $\phiv^*$ agrees with $\phiv$ on every edge incident to $V(T_{y_{p-1}})$, then $T$ is also an ETT of $(G,e, \phiv^*)$.  This observation will be used later in our proof.

Let $T$ be an ETT of a $k$-triple $(G,e, \phiv)$. 
 Let $f_1, \, f_2$, $\dots, f_{n}$ be all the connecting edges of $T$ with  $f_1\prec_{\ell} f_2\prec_{\ell} \dots \prec_{\ell} f_{n}$ and denote $T_i =T_{f_i}-\{f_i\}$ for each $1\le i \le n$. Clearly,  $T_1$ is a maximal Tashkinov tree of $(G, e, \phiv)$ and $T_i$ is closed for every $1\le i\le n$.  We call $T_1\subset T_2 \subset T_3 \subset \dots \subset T_n \subset T$ the {\it ladder of $T$} and $T$ an {\it ETT with $n$ rungs}. We use $m(T)$ to denote the number of rungs of $T$.   Let $D(T) = \{\delta_1, \delta_2, \dots, \delta_n\}$ and $\Gamma(T) =\{\gamma_1, \gamma_2, \dots, \gamma_n\}$ denote the lists of all connecting colors and their companioning colors, respectively.   
 
\begin{DEF}\label{Def-Stable}
 Let $T$ be an ETT of a $k$-triple $(G,e, \phiv)$ with ladder $T_1\subset T_2 \subset T_3 \subset \dots \subset T_n \subset T$. Let $D(T) = \{\delta_1, \delta_2, \dots, \delta_n\}$ and $\Gamma(T) =\{\gamma_1, \gamma_2, \dots, \gamma_n\}$. We say a coloring $\phiv^*\in \CC^k(G-e)$ is {\bf $T$-stable} with regard to $\phiv$ and $T$ is {\bf $\phiv^*/\phiv$-stable} if  the following two conditions are satisfied.
\begin{itemize}
\item   Of the $k$-triple $(G, e, \phiv^*)$, the tree-sequence $T$ is also an ETT with the same sets of connecting edges, connecting colors and companion colors. 
\item  For each $1\leq i \leq n$ and every $f$ incident to $V(T_n)$, 
 $\varphi(f)=\varphi^*(f)$ if $\varphi(f)\in\{\delta_i,\gamma_i\}$ or $\varphi^*(f)\in\{\delta_i,\gamma_i\}$.
\end{itemize}
\end{DEF}
By Definition~\ref{Def-Stable}, we can easily check that: (a) $\phiv$ itself is $T$-stable  with regard to $\phiv$; (b) if $\phiv^*$ is $T$-stable with regard to $\phiv$ then $\phiv$ is $T$-stable with regard to $\phiv^*$;  and (c) if $\phiv^*$ is $T$-stable with regard to $\phiv$ and $\phiv^{**}$ is $T$-stable with regard to $\phiv^*$, then $\phiv^{**}$ is $T$-stable with regard to $\phiv$.  So, all $T$-stable colorings with regard to $\phiv$ form an equivalent class and can be with  regard to any coloring in the class. We call $\phiv^*$ a $T$-stable coloring and $T$ $\phiv^*$-stable for convenience.   
Clearly, if $\phiv^*$ is $T$-stable, then it is $T_x$-stable for any $x$-segment where $x$ is a vertex of $T$. Moreover, we have the following.

\begin{LEM}\label{stable}
Let $T$ be an ETT of a $k$-triple $(G, e, \phiv)$  and $y_p$ be the last vertex of $T$.  If a coloring $\varphi^*\in \CC^k(G-e)$ agrees with $\phiv$  on all edges incident to $V(T-y_p)$, then $\varphi^*$ is $T$-stable.
\end{LEM}
\begin{proof}
Let $T$, $(G, e, \phiv)$ and $\phiv^*$ be defined as in Lemma~\ref{stable}.   Since $\phiv^*$ agrees with  $\phiv$  on 
every edge incident to $V(T_{y_{p-1}})$,  $T$ is an ETT of $(G, e, \phiv^*)$. Let $T_1\subset T_2\subset \dots \subset T_n\subset T$ be the ladder of $T$. Since $T_n\subseteq T_{y_{p-1}}$, $\phiv^*$ agrees with $\phiv$ on every edge incident to $V(T_n)$.  
So, $\phiv^*$ is $T$-stable. \end{proof}

\begin{DEF}\label{Def-MP&R1}
Let $T$ be an ETT of a $k$-triple $(G, e, \phiv)$ with ladder 
$T_1\subset T_2 \subset \dots T_n \subset T$. 
\begin{itemize}
	\item We say that $T$ satisfies condition {\bf MP} (Maximum Property) if $T_1$ is a maximum Tashikov tree and for each $2\leq i \leq n$,  $T_i$ is closed among all $T_i$-stable colorings. 
	\item We say that $T$ satisfies condition {\bf R1} if for each  companion color  $\gamma_i$ of a connecting color $\delta_i$ with $1\leq i\leq n$, $\gamma_i \in \phibar_v(T_{m_i})- \phiv_e(T_{M_i})$, where $m_i$ and $M_i$ are  the minimum and maximum indices, respectively,  such that $\delta_{m_i} = \delta_i = \delta_{M_i}$.  
\end{itemize}

\end{DEF}

\begin{LEM}\label{MPR1}
	Let $T$ be an ETT of a $k$-triple $(G, e, \phiv)$ and $\phiv^*$ be a $T$-stable coloring with regard to $\phiv$.  If $T$ satisfies condition MP (resp. R1) under $\phiv$, then it satisfies condition MP (resp. R1) under $\phiv^*$.

\end{LEM}
\begin{proof} Let $T_1\subset T_2 \subset T_3 \subset \dots \subset T_n \subset T$ be the ladder of $T$ and $\Gamma(T) =\{\gamma_1, \gamma_2, \dots, \gamma_n\}$. Since $\phiv^*$ is $T$-stable, $T_1\subset T_2 \subset \dots T_n \subset T$ is the ladder of the ETT $T$ under $\phiv^*$.   Assume that $T$ satisfies condition MP under $\varphi^*$.  Clearly, $|T_1|$ is still maximum over all $k$-triples.  For each $1\leq i\leq n$, let $\phiv^{**}$ be an arbitrary $T_i$-stable coloring with regard to $\phiv^*$.  Then, it is a $T_i$-stable coloring with regard to $\phiv$. Since $T$ satisfies condition MP under $\phiv$, $T_i$ is closed under $\phiv^{**}$. Therefore, $T$ satisfies condition MP under $\phiv^*$. 
	
Assume that $T$ satisfies condition R1 under $\varphi$.  Then,  $\gamma_i\in \phibar_v(T_{m_i}) - \phiv_e(T_{M_i})$. Therefore $\gamma_i\in \phibar_v(T_{m_i})$ and $\gamma_i\notin \phiv_e(T_{M_i})$.  Since $\phiv^*$ is $T$-stable with regard to $\phiv$, $\phiv^*$ and $\phiv$ have the same set of $\gamma_i$ edges incident to $V(T_{M_i})$.  Hence $\gamma_i\in \phibar^*_v(T_{m_i})$ and $\gamma_i\notin \phiv^*_e(T_{M_i})$. Thus we have $\gamma_i \in \phibar^*_v(T_{m_i}) - \phiv^*_e(T_{M_i})$.  Therefore $T$ still satisfies condition R1 under $\varphi^*$.
\end{proof}

Let  $(G, e, \phiv)$ be a $k$-triple and $T$ be an ETT of $G$.  We call the algorithm of adding a boundary edge $f$ and $b(f)$ to $T$ with $\phiv(f) \in \phibar_v(T)$ Tashkinov Augment Algorithm ({\bf TAA}). 
Given an ETT with ladder $T_1\subset T_2 \subset \dots \subset T_n \subset T$, we note that conditions MP and R1 only apply to $T_i$ with $i \le n$.  So, the following result holds.

\begin{LEM}\label{bclear}
	Let $T$ be an ETT of a $k$-triple $(G,e,\varphi)$  satisfying conditions MP and R1. If $T'$ is an ETT obtained from $T$ by adding some new edges and vertices by TAA under $\varphi$, then $T'$ also satisfies conditions MP and R1 under $\varphi$.
\end{LEM}

The following is the main theorem of this paper.
\begin{THM}\label{bmain}
	Let $T$ be an ETT of a $k$-triple $(G,e,\varphi)$ with $G$ being non-elementary. If $T$ satisfies conditions MP and R1 under $\varphi$, then $T$ is elementary.
\end{THM}	

Note that if $m(T) = 0$, then $T$ is a Tashkinov tree, so it satisfies conditions MP and R1 by default. If $m(T)=1$ and $T_1$ is a maximum Tashkinov tree, then $T$ also satisfies both conditions
MP and R1.  

\begin{COR}\label{bT2}
 Let $T$ be a closed ETT of a $k$-triple $(G, e, \phiv)$ with $G$ being non-elementary. If $T$ satisfies MP and  all its connecting colors are distinct, $T$ is elementary. In particular, if $m(T)=1$ and $T_1$ is a maximum Tashkinov tree,  then $T$ is elementary.
\end{COR}	
\begin{proof}
We only need to verify that condition R1 is satisfied. Since all companion colors $\gamma_1$, $\gamma_2$, $\dots$, $\gamma_n$ are distinct, $m_i= M_i =i$ for each $1\le i\le n$. By the definition of connecting edge of $T_i$,  we have $\gamma_i\in \phibar_v(T_i) - \phiv_e(T_i)$. 
\end{proof}

In application, we will use the following result. Then stronger version of Theorem~\ref{bmain2a-0} will be given as Theorem~\ref{bmain2a} in Section~\ref{4}, and its proof is based on Theorem~\ref{bmain}.

\begin{THM}\label{bmain2a-0}
	Let $G$ be a $k$-critical graph with $k \ge \D +1$.  If $G$ is not elementary, then there exist a $k$-triple $(G, e, \phiv)$, a maximum Tashkinov tree $T_1$ and an elementary ETT $T \supset T_1$  such that the following hold.
	\begin{eqnarray}
    |T - T_{1}| & \ge & 2 |\phibar_v(T_{1})| +2 \label{beqn-(T-Tn)2-0}\\
	|T-T_1| & > & 2(1+\frac{\chi'-1-\Delta}{\mu})^{|\phibar_v(T_1)|}\label{beqn-(T-Tn)3-0}
	\end{eqnarray}
\end{THM}



\begin{LEM}{\em [Scheide~\cite{Scheide-2010}]}\label{bCOR:T-order-exit-vertex}
	Let $G$ be a $k$-critical graph with $k\ge \Delta +1$.  If $G$ is not elementary,    then 
	$|T|\ge max\{2(k-\Delta)+1,11\}$ for every maximum Tashkinov tree $T$ of $G$.  
\end{LEM}
Since we mainly work on non-elementary graphs in this paper, we assume that $|T|\geq 11$ for every maximum Tashkinov tree $T$ by Lemma~\ref{bCOR:T-order-exit-vertex}. Hence we have $|\phibar_v(T)|\geq 13$ if $T$ is a maximum Tashkinov tree, because $e$ is uncolored.
\begin{LEM}\label{bTHM:main3}
	If $G$ is a non-elementary $k$-critical graph $G$ with $k\ge \D +1$, then there exists a $k$-triple $(G, e, \phiv)$ and  an elementary ETT  $T$ such that  
	\begin{equation}
	|T| \ge \max\{ (2(k-\D) +1)^2+6,  22(k-\D)+17\}\geq 39.
	\end{equation}
	
\end{LEM}
\begin{proof}
By Theorem~\ref{bmain2a-0},  there exist a $k$-triple $(G, e, \phiv)$,  a maximum Tashkinov tree $T_1$ and an elementary ETT $T \supset T_1$  such that  the following holds. 
	\begin{equation*} 
	\begin{aligned}
	|T|&\geq 2|\phibar_v(T_1)|+|T_1|+2\\&\geq 2(|T_1|(k-\Delta)+2)+|T_1|+2\\&\geq 2 (\max\{2(k-\Delta)+1,11\}(k-\Delta)+2)+\max\{2(k-\Delta)+1,11\}+2\\&\geq\max\{ (2(k-\D) +1)^2+6,  22(k-\D)+17\}.
	\end{aligned}
	\end{equation*}
\end{proof}

The following result gives an improvement a result of
Chen et al.~\cite{ChenGKPS16+} that
 if $\chi' \ge \D +\sqrt[3]{\D/2}-1$ then $\chi' = \rho$.
\begin{THM}\label{bTHM:cubic}
If $G$ is a graph with $\D \ge 5$ and    $\chi'\geq\Delta+\sqrt[3]{\Delta/4}$,
then $G$ is elementary.
\end{THM}
\begin{proof}
We assume without loss of generality $G$ is critical, but not elementary. 
Since $\D \ge 5$, $\chi' \ge \D + \sqrt[3]{\D/4} \ge \D +2$.  By Lemma~\ref{bTHM:main3},  there exists a $k$-triple $(G, e, \phiv)$ and  an  elementary ETT 
	$T$ such that  $|T| \ge (2(k-\D) +1)^2+6$.  Since $T$ is elementary,  we have
\[
k\geq\phibar_v(T)\ge ((2(k-\D) +1)^2+6)(k-\Delta)+2,
\]
which  gives $k<\Delta+\sqrt[3]{\Delta/4} -1$,  a contradiction. 
\end{proof}

We now show that Conjecture~\ref{con:Jm} is true up to $m =39$.  The following  observation from~\cite{StiebSTF-Book} is needed. For completeness, we give its proof here.

\begin{LEM}\label{bLEM:elementary-Jm}
If $(G, e, \phiv)$ be a $k$-triple with 
$k>   \frac{m}{m-1}\Delta+\frac{m-3}{m-1}-1$, 
then $|X|\le m-1$
	for 
	every elementary set $X\subseteq V(G)$ with $V(e)\subseteq X$. 
\end{LEM}
\proof Suppose on the contrary $|X| \ge m$.  The  inequality $k>   \frac{m}{m-1}\Delta+\frac{m-3}{m-1}-1$ is equivalent to 
$k -  \D >\frac{\D -2}{m-1}$. 
Since $X$ is elementary,    $k\ge |\phibar(X)| \ge (k-\Delta)|X| +2 \ge m(k-\Delta) +2$.  
Subtracting $k-\D$, we get  $\Delta \ge (m-1)(k-\Delta) +2 > (\D -2)+2 = \Delta$,  
a contradiction. \qed

\begin{THM}\label{bTHM:Jm35}
If  $G$ is a graph with $\chi'>\frac{39}{38}\Delta+\frac{36}{38}$,
then $G$ is elementary.
\end{THM}
\begin{proof}
Otherwise, by Lemmas~\ref{bTHM:main3}, $G$ has an elementary set $X$ with $|X| \ge 39$, which gives a contradiction to 
Lemma~\ref{bLEM:elementary-Jm}.
\end{proof}

\begin{COR}\label{bCOR:Delta35}
Let $G$ be a graph with $\chi' \ge \D +2$. If $\Delta\le 39$ or $|V(G)|\le 39$,  then $G$ is elementary.
\end{COR}
\begin{proof} 
When $\D \le 39$, we have $\D +2 \ge \frac{38}{39}\D + \frac{36}{38}$, so $G$ is elementary by Theorem~\ref{bTHM:Jm35}.  If $G$ is not elementary, then $G$ contains an elementary ETT $T$ with $|T|\geq 39$ by Lemma~\ref{bTHM:main3}, and therefore $|V(G)|>39$. So $G$ is elementary if $|V(G)|\leq 39$.
\end{proof}
Haxell and McDonald~\cite{HaxellM11} obtained a necessary and sufficient condition for  $\chi' = \Delta + \mu$ when $\mu \ge \log_{5/4}\D +1$.  We improve this result lowing the lower bound of $\mu$.
\begin{THM}\label{bTHM-mu}
If $G$ is  a graph with multiplicity $\mu \ge \log_{\frac{5}{4}}\left(\log_{\frac{3}{2}}(\frac{\D}{2})\right) +1$, then $\chi' = \D + \mu$ if and only if $\rho = \D + \mu$, where $\rho$ is the density of $G$. 
\end{THM}
\proof Since $\rho\leq\chi'\leq \D+\mu$ as mentioned earlier, we have $\chi' = \D + \mu$ if $\rho = \D + \mu$. We now suppose
$\mu \ge \log_{5/4}\left(\log_{3/2}(\frac{\D}{2})\right) +1$ and $\chi' = \D + \mu$.  To show $\rho' = \D + \mu$, we only need to show that $\rho=\chi'$, i.e., $G$ is elementary.  Assume without loss of generality that $G$ is critical. Suppose on the contrary $G$ is not elementary.  By (2) of Theorem~\ref{bmain2a-0}, there exists an elementary ETT $T$ containing a maximum Tashkinov tree $T_1$ having the following property:
\begin{equation}\label{1}
|T-T_1| \ge 2 \left(1 + \frac{\chi' -1 -\D}{\mu}\right)^{|\phibar_v(T_1)|} >
2 \left(1+ \frac{\mu -1}{\mu}\right)^{(\mu -1)|T_1|+2}
\end{equation}
Here we have $\chi' -1 -\D=\mu-1$ and $|\phibar_v(T_1)|>(\mu-1)|T_1|+2$ because of the assumption that $\chi' = \D + \mu$ and the fact that $T_1$ is elementary as a Tashkinov tree.
 Hexell and McDonald~\cite{HaxellM11} gave a lower bound of $|T_1|$ bellow: 
 \begin{equation}\label{2}
 |T_1| \ge \left(1+\frac{\chi'-1-\Delta}{2\mu}\right)^{\chi'-1-\Delta}+1
 \ge \left(1 +\frac{\mu -1}{2\mu}\right)^{\mu -1} +1.
 \end{equation}
 Note that  $|\phibar_v(T)| \le \chi' -1$.
  Since $T$ is elementary, we have 
 $|\phibar_v(T)|\geq (\mu-1)|T|$. Recall that $\chi' -1 -\D=\mu-1$, we have 
$|T|(\chi'-1-\Delta)\leq\chi'-1$, so
 $(|T|-1)(\chi'-1-\Delta)\leq\Delta$, which is equivalent to \[
  (|T|-1)\leq \frac{\Delta}{\mu-1}.
\]
Since $|T|-1=|T-T_1|+|T_1|-1\geq |T-T_1|$, by~(\ref{1}) and~(\ref{2}) we have
\[ 2\left(1 +\frac{\mu -1}{\mu}\right)^{(\mu -1)(1 +\frac{\mu -1}{2\mu})^{\mu -1}}<\frac{\Delta}{\mu-1}.
\]
Hence
\[ 2(\mu-1)\left(1 +\frac{\mu -1}{\mu}\right)^{(\mu -1)(1 +\frac{\mu -1}{2\mu})^{\mu -1}}<\D.
\]

 Note that $\mu \ge 2$ and $2(\mu-1)\left(1 +\frac{\mu -1}{\mu}\right)^{(\mu -1)(1 +\frac{\mu -1}{2\mu})^{x -1}}$ is an increasing function of $\mu$ when we fix $x$ with $x\geq 2$, we get $\frac{3}{2}^{\frac{5}{
 		4}^{\mu-1}}<\frac{\D}{2}$ by plugging in $\mu=2$. Thus we have $\mu < \log_{\frac{5}{4}}\left(\log_{\frac{3}{2}}(\frac{\D}{2})\right)+1$, a contradiction.  \qed

Haxell and McDonald in the same paper proved that a graph is elementary if $\chi' \ge \D + 2\sqrt{\mu \log \D}$, where $\log$ denotes the natural logarithm.  We improve their result as follows.
\begin{THM}
Let $G$ be a graph with $\chi'>\D+1$. 
Then $G$ is elementary if 
$\chi' \geq \Delta+\min\{2\sqrt{\mu(\log\log\frac{\D}{2}+\log2\mu)},
\sqrt[3]{\mu\log\frac{\D}{2}}\}$.

\end{THM}	
\proof Let $G$ be a graph satisfying the above conditions and assume on the contrary
$G$ is not elementary.  Assume without loss of generality that $G$ is critical.  Let $t=\chi'-1-\D$ and $T$, $T_1$ be as defined in Theorem~\ref{bmain2a-0}. Following similar arguments as in the proof of Theorem~\ref{bTHM-mu}, we have the following inequality:
\begin{equation}
t \cdot (|T-T_1| + |T_1|) +2 \le |\phibar_v(T)| \le \chi' -1 = \D +t \label{bcomb-3}
\end{equation}
 Combine (\ref{1}) and (\ref{2}) with (\ref{bcomb-3}),   we have 
 \[
 t\left(1+\frac{t}{\mu}\right)^{t\left(1+\frac{t}{2\mu}\right)^{t}}<\frac{\Delta}{2}.
 \]
Since $0<t/\mu<1$, we have $1+t/2\mu>e^{t/4\mu}$ and $1+t/\mu>e^{t/2\mu}$, which in turn gives $te^{\frac{t^2}{2\mu}e^{\frac{t^2}{4\mu}}}<\frac{\Delta}{2}$. Since $t\geq 1$, we have
$t <  2\sqrt{\mu(\log\log\frac{\D}{2}+\log2\mu)}$ by plugging $t=1$ into $te^{\frac{t^2}{2\mu}}$. 
By Lemma~\ref{bCOR:T-order-exit-vertex}, we have $|T_1| \ge 2t +1$.  Using this inequality with (\ref{1}) and (\ref{bcomb-3}), we have $t\left(1+\frac{t}{\mu}\right)^{t(2t+1)+2}<\frac{\D}{2}$. Since $1+t/\mu>e^{t/2\mu}$ and $t\geq 1$, we similarly have $t<  \sqrt[3]{\mu\log\frac{\D}{2}}$. Thus we have $t< \min\{2\sqrt{\mu(\log\log\frac{\D}{2}+\log2\mu)},
\sqrt[3]{\mu\log\frac{\D}{2}}\}$, giving a contradiction.
\qed

\section{Condition R2}\label{pre}

Let $(G, e,\phiv)$ be a $k$-triple.  For a color set $B$ and a subgraph $H\subseteq G$, we call $H$ is $B$-closed if $\phiv(\partial(H)) \cap B = \emptyset$ and
$H$ is $B^-$-closed if $H$ is $(\phibar_v(H)-B)$-closed.  Clearly, $H$ is closed if $H$ is $\phibar_v(H)$-closed. When $B=\{\beta\}$ is a singleton, we say $H$ is $\beta$-closed if it is $\{\beta\}$-closed. We also say $\beta$ is closed in $H$ and $H$ is closed for $\beta$ for convenience if $H$ is $\beta$-closed.  Let $T$ be an ETT with ladder $T_1\subset  T_2 \subset \dots \subset T_n\subset T$ of a $k$-triple $(G, e,\phiv)$.  We call the subsequence $T-T_n$ the {\it tail} of $T$ and any nested sequence of segments $T_{n,0}(=T_n) \subset T_{n,1}\subset \dots \subset T_{n,q}\subset T_{n,q+1} (=T)$ a {\it split tail} for $T$ if each $T_{n, j}$ ends with a vertex of $T -T_n$.  We further call the sequence $T_1\subset  T_2 \subset 	\dots \subset T_n:=T_{n,0} \subset T_{n,1}\subset 	\dots \subset T_{n,q}\subset T:=T_{n,q+1}$ a {\it refinery of $T$ with $n$ rungs and $q$ splitters}, or simply a {\it refinery} of $T$. For each $T_{n,j}$ with $0\leq j\leq q$, let $D_{n,j}=D(T)-\phibar_v(T_{n,j})=\{\delta_1,...,\delta_n\}-\phibar_v(T_{n,j})$. Clearly, $D_{n,q+1}\subseteq D_{n,q}\subseteq,...,\subseteq D_{n,0}$.

\begin{DEF} \label{Defi-R2} 
An ETT $T$  satisfies condition {\bf R2} if  $T$ has a refinery 
$$
T_1\subset  T_2 \subset \dots \subset T_n=T_{n,0} \subset T_{n,1}\subset 	\dots \subset T_{n,q}\subset T_{n,q+1} = T
$$
such that 
for every  $0 \le j \le q$ and every $\delta_h\in D_{n,j}$, there exists a two color set $\Gamma^{j}_h=\{\gamma^{j}_{h_1},\gamma^{j}_{h_2}\}$ satisfying 
the following four properties. 

\vspace{-0.2in}
\begin{enumerate}
\item $\Gamma^{j}_h\subseteq\phibar_v(T_{n,j})-\varphi_e(T_{n,j+1}(v(\delta_h))-T_{n,j})$ for every $j$ and $\delta_h\in D_{n, j}$.
		
\item  $\Gamma^{j}_{h}\cap\Gamma^{j}_{g}=\emptyset$ for every $j$ and two distinct $\delta_h,\delta_g\in D_{n,j}$.
\item $\Gamma^{j}-\Gamma^{j-1}\subseteq\phibar_v(T_{n,j}-T_{n,j-1})$ for each $1\leq j\leq q$, \\ where $\Gamma^{j}=\cup_{\delta_h\in D_{n,j}}\Gamma^{j}_h$  and $\Gamma^{j-1}=\cup_{\delta_h\in D_{n,j-1}}\Gamma^{j-1}_h$.
		
\item $T_{n,j}$ is $(\cup_{\delta_h\in D_{n,j}}\Gamma^{j-1}_h)^-$-closed for every $1\le j \le q$. 
\end{enumerate}	
\end{DEF}

\begin{REM}
{\em Not every ETT $T$ satisfies condition R2. We will show in statement B of Section~\ref{bmainproof} that for every $T$ satisfying conditions MP and R1, there exists an ETT $T'$ with $V(T') \supseteq V(T)$ satisfying conditions MP, R1 and R2. Since switching colors $\delta_i$ with another color on a color alternating chain usually creates a non-stable coloring, we may use colors in $\Gamma^{j}_h$ as stepping stones to swap colors while keeping the coloring stable in later proofs. Thus, we may consider the set $\Gamma^{j}_h$ as a color set reserved for $\delta_h$  and (1) as a condition that ensures the ETT properties after some changes of colorings.  We also notice that (1) and (2) actually involve $T_{n,q+1}$ for $j=q$ while (3) and (4) only involve $T_{n,q}$.  	
}		
	
\end{REM}

Let $T$ be an ETT of $(G,e,\varphi)$ satisfying condition R2. In the remainder of the proof, when we mention that $T$ satisfies condition R2 under another coloring $\varphi^*$ associated with $\varphi$, we always mean that $T$ satisfies condition R2 under $\varphi^*$ with the same refinery and $\Gamma_h^{j}$ as under $\varphi$ for every $0\leq j\leq q$ and every $\delta_h\in D_{n,j}$. Let $\alpha,\beta$ be two colors and $Q$ be an $(\alpha,\beta)$-chain.  If $V(Q)\cap V(T)\neq\emptyset$, we say $Q$ {\it intersects} $T$.



\begin{LEM}\label{bLEM:Stable}
Let $T$ be an ETT of a $k$-triple $(G, e, \phiv)$ and  $\phiv^*$  be obtained from $\phiv$ by recoloring some $(\alpha, \beta)$-chains. If these $(\alpha, \beta)$-chains do not intersect $T-y_p$,  then $T$ is $\phiv/\phiv^*$-stable.  Moreover,  if $T$ additionally satisfies condition MP (resp. R1, resp. R2)  under $\varphi$, then it also satisfies condition MP (resp. R1, resp. R2)  under $\varphi^*$.
\end{LEM}
\proof  
	Since $\varphi^*$ and $\varphi$ agree on all edges incident to $V(T-y_p)$,  $\varphi^*$ is a $T$-stable coloring by Lemma~\ref{stable}.
Moreover,  if $T$ satisfies condition MP (resp. R1 ) under $\varphi$, then it satisfies condition MP (resp. R1) under $\varphi^*$ by Lemma~\ref{MPR1}. 
Assume that $T$ satisfies condition R2 under $\varphi$. Let the corresponding splitting tail of $T$ be $T_n=T_{n,0} \subset T_{n,1}\subset 	\dots \subset T_{n,q}\subset T=T_{n,q+1}$. Since $\phiv^* (f) = \phiv^*(f)$ for every edge $f$ incident to $V(T-y_p)$, $\phibar(v)=\phibar^*(v)$ for every vertex $v\in T-y_p$ and $T_{n,i}\subseteq T-y_p$ for every $0\leq i\leq q$,  the conditions (1), (2), (3) and (4) in Definition~\ref{Defi-R2} are satisfied for $T$ under $\varphi^*$ with the same $\Gamma_h^{i}$ as under $\varphi$ for each $0\leq i\leq q$ and each $\delta_h\in D_{n,i}$. Therefore, $T$ also satisfies condition R2 under $\varphi^*$.
 \qed

\section{An applicable result}\label{4}

An ETT $T$ of a $k$-triple $(G, e, \phiv)$ with ladder $T_0\subset T_1 \subset \dots \subset T_n \subset T$ is called a simple ETT (SETT) if $\gamma_1 = \gamma_2 = \dots = \gamma_n$.   By the definition of companion colors, $\gamma_i\in \phibar_v(T_i) -\phiv_e(T_i)$ for each $1\le i \le n$.  So, we have $\gamma_1\in \phibar_v(T_1) - \phiv_e(T_n)$ if $T$ is an SETT, which in turn shows that  all SETTs satisfy condition R1.   Let $\alpha$ and $\beta$ be two colors and suppose $T$ is $\{\alpha, \beta\}$-closed under $\phiv$. Let $\phiv/(G-T, \alpha, \beta)$ be a coloring obtained from $\phiv$ by interchanging these two colors outside $T$.  Clearly, $\phiv/(G-T, \alpha, \beta)$ is also a $k$-edge-coloring. By Lemma~\ref{stable}, $\phiv/(G-T, \alpha, \beta)$ is $T$-stable. We prove the following result  which is a stronger version of Theorem~\ref{bmain2a-0} in Section 2. 
\begin{THM}\label{bmain2a}
	Let $G$ be a $k$-critical graph with $k \ge \D +1$.  If $G$ is not elementary, then there exist a $k$-triple $(G, e, \phiv)$, a maximum Tashkinov tree $T_1$ and an elementary SETT $T \supset T_1$ satisfying condition MP with the following three properties:
	\begin{eqnarray}
	\phibar_v(T_n) & \subseteq & \phiv_e(T-T_n) \label{beqn-(T-Tn)1}\\
	|T - T_{n}| & \ge & 2 |\phibar_v(T_{n})| +2 \label{beqn-(T-Tn)2}\\
	|T-T_n| & > & 2(1+\frac{\chi'-1-\Delta}{\mu})^{|\phibar_v(T_n)|}\label{beqn-(T-Tn)3}
	\end{eqnarray}
\end{THM}

\proof 
Let $G$ be a non-elementary $k$-critical graph with $k \ge \D +1$.
We first note that every maximum Tashkinov $T$ is closed under any $T$-stable coloring since, otherwise, there would be a larger Tashkinov tree. Moreover,  every Tashkinov tree is an SETT by default.  Based on the above statements, we let $T$ be an SETT satisfying condition MP with ladder $T_1\subset T_2 \subset,...,\subset T_n \subset T$  such that $T$ is closed under every $T$-stable coloring.  We further assume that $m(T)=n$ is maximum.   Let  $\gamma := \gamma_1
=\gamma_2 \dots = \gamma_n$.   Since all SETTs satisfy condition R1,  $T$ satisfies both conditions MP and R1. By Theorem~\ref{bmain},  $T$ is elementary.   Since $G$ is not elementary, $T$ is not strongly closed. So, $T$ has a defective color. 

We first show that (\ref{beqn-(T-Tn)1}) holds. Otherwise, let $\alpha \in  \phibar_v(T_n) - \phiv_e(T- T_n)$ and $\phiv^* = \phiv/(G-T_n, \alpha, \gamma)$. By Lemma~\ref{bLEM:Stable}, $\phiv^*$ is a $T_n\cup\{f_n,b(f_n)\}$-stable coloring. Since $\alpha,\gamma\in\phibar_v(T_n)$, $T$ is still an ETT under $\varphi^*$ with the same set of connecting colors and companion colors under $\varphi$. Therefore $\varphi^*$ is $T$-stable. Since $\alpha\notin\phiv_e(T- T_n)$, $\gamma\notin \phiv^*_e(T-T_n)$, which in turn shows 
$\gamma\notin\varphi^*_e(T)$ because $\gamma\in \phibar_v(T_1) - \phiv_e(T_n)$.  Let $\delta_{n+1}$ be a defective color of $T$.  Since $T$ is elementary, $T$ can not contain both ends of $P_{v(\gamma)}(\delta_{n+1},\gamma,\varphi^*)$. Since $\gamma\in \phibar^*_v(T_1) - \phiv^*_e(T)$,  we can extend $T$ to a larger SETT $T^*$ by adding a connecting edge $f_{n+1}$ which is the first edge in $\partial(T)$ along $P_u(\delta_{n+1},\gamma,\varphi^*)$. Moreover, the resulting SETT satisfies condition MP because we assumed that $T$ is closed among all $T$-stable colorings.  Since stable colorings keep conditions MP and R1 by Lemma~\ref{MPR1}, by taking maximum value of $|T^*|$  with the above properties, we can assume $T^*$ is closed under all $T^*$-stable colorings, which gives a contradiction to the maximality of $m(T)$. 

Recall $T$ is elementary as we mentioned earlier.  To prove  (\ref{beqn-(T-Tn)2}) and (\ref{beqn-(T-Tn)3}), starting from $T_n\cup \{f_n\}$ 
we apply TAA   repeatedly with priority of adding boundary edges with colors being missing on the vertices not in $T_n$ and call such an algorithm  {\it modified TAA} (mTAA).   Clearly, the resulting closed SETT has the same vertex set as $T$.  Assume, without loss of generality, $T$ itself is the resulting tree by applying mTAA till $T$ is maximal to get a closed extension of $T_n\cup \{f_n\}$. For any $\alpha\in \phibar_v(T_n)$, let $T_{\alpha}$ be the maximal segment of $T$ before the last edge with color $\alpha$ being added, i.e, the last element of $T_{\alpha}$ is the vertex before the last edge colored by $\alpha$ in $T$ along $\prec_{\ell}$.
By (\ref{beqn-(T-Tn)1}), $T_{\alpha}$ is a proper subtree of $T$ for each $\alpha\in \phibar_v(T_n)$. Moreover, we have $V(T_{\alpha}-T_n)\neq\emptyset$ for each $\alpha\in \phibar_v(T_n)$ since the last connecting edge $f_n$ is not colored by colors missing in $T_n$.   

We claim that $|T_{\alpha}|$ is odd for each $\alpha\in \phibar_v(T_n)$.  Otherwise, we assume $|T_{\alpha}|$ is even and let $\beta \in \phibar_v(T_{\alpha}-T_n)$. Since $T$ is elementary, $T_{\alpha}$ is also elementary. Since $T_{\alpha}$ is elementary and has even number of vertices, it has a boundary edge  colored by $\beta$ which should be added to $T_{\alpha}$ before the corresponding $\alpha$-edge, a contradiction. By~(\ref{beqn-(T-Tn)1}), each color $\alpha\in\phibar(T_n)$ must be used by an edge in $T-T_{n}$. Since each color $\alpha\in\phibar_v(T_n)$ is used by an edge in $T-T_{n}$ by~(\ref{beqn-(T-Tn)1}) and $|T_{\alpha}|$ is odd for each $\alpha\in \phibar_v(T_n)$, we have $|E(T-T_n)| \ge 2 |\phibar_v(T_n)| +2$ , where additional $2$  follows from $\phiv(f_n) \notin \phibar_v(T_n)$. So, (\ref{beqn-(T-Tn)2}) holds.  

To prove (\ref{beqn-(T-Tn)3}), let $e_1\prec_{\ell}e_2\prec_{\ell}\dots\prec_{\ell}e_q$ be the list of all edges in $T -T_n$ such that $\phiv(e_i)\in \phibar_v(T_n)$ for all $i=1,2, \dots, q$.  Clearly, $q \ge |\phibar_v(T_n)|$. 
For each $1 \le i \le q$,  since $T$ is elementary we have $\varphi(b(e_i)) \supseteq \phibar (T_{e_i} - T_n)$.
On the other hand, according to mTAA, 
$\phiv(\partial(T_{e_i} ) \cap \phibar_v(T_{e_i} - T_n) = \emptyset$.  By eliminating parallel edges,  we can add  at least $|\phibar_v(T_{e_i} - T_n)|/\mu$ neighbors of    
$b(e_i)$ in $V-V(T_{e_i})$ to $T_{e_i}\cup\{b(e_i)\}$ using colors in $\phibar_v(T_{e_i} - T_n)\ne \emptyset$.  Since $|\phibar(v)| \ge \chi' -1 -\D$ for all $v\in V(G)$,   the following inequalities hold:
\[
|T_{e_{i+1}} -T_n| \ge 
1+ |T_{e_i} -T_n|  + \frac{|\phibar_v(T_{e_i} -T_n)|}{\mu} >  
|T_{e_i} -T_n|(1 +  \frac{\chi' -1 - \D}{\mu})
\]
where $T_{e_{q+1}} = T$. Hence, 
\[
|T - T_n| \ge |T_{e_1}-T_n| (1 + \frac{\chi' -1 - \D}{\mu})^q 
\ge |T_{e_1}-T_n| (1 + \frac{\chi' -1 - \D}{\mu})^{|\phibar_v(T_n)|}.
\]
Note that $T_{e_1}$ contains $f_n$ and one more edge with color in
$\phibar(b(f_n))$, we have $|T_{e_1} - T_n| \ge 2$, which in turn gives 
(\ref{beqn-(T-Tn)3}). \qed  

\section{Proof of Theorem~\ref{bmain}}\label{bmainproof}


\setcounter{section}{2}
\setcounter{THM}{4}
\begin{THM}
	Let $T$ be an ETT of a $k$-triple $(G,e,\varphi)$ with $G$ being non-elementary. If $T$ satisfies conditions MP and R1 under $\varphi$, then $T$ is elementary.
\end{THM}	

\setcounter{section}{5}
\setcounter{THM}{0}

	Let $T$ be an ETT of a $k$-triple $(G,e,\varphi)$ with $G$ being non-elementary.  We will prove Theorem~\ref{bmain} inductively on $m(T)$,  the number of rungs. To facilitate our proof, we add the following two  statements  simultaneously for each nonnegative integer $n$.  

\begin{enumerate}
	\item[A.]  If $T$ is an ETT satisfying conditions MP, R1 and R2 with $m(T)=n$, then $T$ is elementary.
	\item[B.]  Suppose statement A holds. If $T$ is an ETT with ladder $T_1\subset T_2 \subset \dots \subset T_n \subset T$ satisfying conditions MP and R1, then there exists a closed ETT $T'$ with $V(T)\subseteq V(T')$ and ladder $T_1 \subset T_2 \subset \dots \subset T_n \subset T'$ satisfying conditions MP, R1 and R2. 
	
\end{enumerate}	
Although statement A appears weaker than Theorem~\ref{bmain}, statement B shows that they are equivalent. By Tashkinov's Theorem, Theorem~\ref{bmain} holds for $n=0$. Assume $n\ge 1$ and all ETTs with $m(T)\le n-1$ satisfying MP and R1 are elementary.  We will show both statements hold for ETTs $T$ with $m(T) = n$, and consequently, they imply all ETTs with $m(T)=n$ satisfying MP and R1 are elementary based on following: Let $T$ be an ETT with $m(T) =n$ satisfying conditions MP and R1. By statement B, there exists an ETT $T'$ with $m(T')=n$ satisfying MP, R1 and R2 such that $V(T') \supseteq V(T)$.  By statement A, $T'$ is elementary, so is $T$. 

The following flowchart depicts the proof strategy. We place the proof of statement B first since it is much shorter than the proof of statement A. 
\begin{figure}[!htb]
	\begin{center}

		\includegraphics[height=3.5cm]{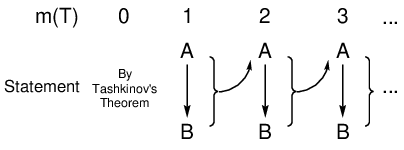}
		
	\end{center}
\end{figure}

\subsection{Proof of statement B}


\proof
Assume that statement A holds for all ETTs with at most $n$ rungs which satisfy conditions MP, R1 and R2.
Let $T$ be an ETT with ladder $T_1 \subset T_2 \subset \dots \subset T_n \subset T$ satisfying conditions MP and R1 of  a $k$-triple $(G,e,\varphi)$.
Starting with $T_n$, we will construct an ETT $T'$ with $V(T)\subseteq V(T')$ by adding  a split tail $T_n=T_{n,0} \subset \dots \subset T_{n,q} \subset T'=T_{n,q+1}$ satisfying conditions MP, R1 and R2. 

We first note a simple fact: under the same coloring $\phiv$,  for any ETT $T^*$ with $T^* \supset T_n\cup\{ f_n\}$, if $T^*$ is closed then $V(T^*) \supseteq V(T)$. Otherwise, let  $e_i$ be the first edge of $T$ crossing the boundary of $T^*$. Then, $\phiv(e_i) \in \phibar_v(T^*)$, giving a contradiction to $T^*$ being closed. 

We define $\Gamma^{0}_{h}=\{\gamma^{n,0}_{h1},\gamma^{n,0}_{h2}\}\subseteq\phibar_v(T_{n,0})$ for all $\delta_h\in D_{n,0}$ as follows.    Recall $D(T) = \{\delta_1, \delta_2, \dots, \delta_n\}$  and $D_{n,0} = D(T) - \phibar_v(T_{n,0})$.  So,  $|D_{n,0}| \le n$. Since $T_1$ is a maximum Tashkinov tree,  $|T_1| \ge 11$.  Because $m(T_n)=n-1$, $T_n$ is elementary by the induction hypothesis. For each $1\leq i\leq n$, since $T_i$ is closed, $|T_i|$ is odd.  Thus $|T_n| \ge 2n$, so $|\phibar_v(T_n)| \ge |T_n| +2 \ge 2|D_{n, 0}|$.  We simply pick $2|D_{n,0}|$ colors from $\phibar_v(T_n)$, divide them into $|D_{n,0}|$ pairs,  assign a distinct pair to each distinct connecting color $\delta_h\in D_{n, 0}$ and name it as $\Gamma^0_h:=\{\gamma^{n,0}_{h1}, \gamma^{n,0}_{h2}\}$. Then condition R2(2) is satisfied.  Let $\Gamma_0 = \cup_{\delta_h\in D_{n,0}}$.

We define  $T_{n,1}$ by the following greedy algorithm. We first let $T_{n,1}=T_{n,0}\cup\{f_n,b(f_n)\}$, where $f_n$ is the connecting edge of $T$ after $T_n$.  If there exists an edge $f\in\partial(T_{n,1})$ such that $\varphi(f)\in\phibar_v(T_{n,1})-\cup_{\delta_h\in D_{n,1}}\Gamma^{0}_h$, we augment $T_{n,1}$ by letting $T_{n,1}:=T_{n,1}\cup\{f,b(f)\}$, and we repeat this process until we can not add any new edge, i.e.,  $T_{n,1}$ is $(\cup_{\delta_h\in D_{n,1}}\Gamma^{0}_h)^-$-closed. Note that condition R2(1) and (4) are satisfied through this algorithm.

Suppose $T_{n,j-1}$ is defined for some $j\geq 2$. If $T_{n,j-1}$ is closed, then $V(T)\subseteq V(T_{n,j-1})$ and we let $T_{n,j-1}=T'$. Suppose $T_{n,j-1}$ is not closed.  Since $T_{n, j-1}$ is $(\cup_{\delta_h\in D_{n,j-1}})\Gamma_h^{j-2})^-$-closed, there exists 
an edge $f'\in\partial(T_{n,j-1})$ such that $\varphi(f')\in\Gamma^{j-2}_{h}$ for some $\delta_h\in D_{n,j-1}$.  Let $\Gamma^{j-1}_{h}$ be a set of two missing colors in $\phibar_v(T_{n,j-1}-T_{n,j-2})$,  and $\Gamma^{j-1}_{h^*}=\Gamma^{j-2}_{h^*}$ for any other  $\delta_{h^*}\in D_{n,j-1}$. By statement A, $T_{n,j-1}$ is elementary. Therefore, $|\phibar_v(T_{n,j-1}-T_{n,j-2})|\geq 2$. So, $\Gamma^{j-1}_{h}$ is well defined. Note that condition R2(2) and (3) are satisfied by our choice.

We define $T_{n, j}$ as follows.  We first let $T_{n,j}=T_{n,j-1}\cup\{f',b(f')\}$,  where $f'$ is defined above.   If there exists $f\in\partial(T_{n,j})$ such that $\varphi(f)\in\phibar_v(T_{n,j})-\cup_{\delta_h\in D_{n,j}}\Gamma^{j-1}_h$, we augment $T_{n,j}$ by letting $T_{n,j}:=T_{n,j}\cup\{f,b(f)\}$, and we repeat this process until we can not add any new edge, i.e.,  $T_{n,j}$ is $(\cup_{\delta_h\in D_{n,j}}\Gamma^{j-1}_h)^-$-closed. Then condition R2(1) and (4) are satisfied through this algorithm.

Clearly, $T_{n,j}$  satisfies condition R2. Since $T_{n}\cup\{f_n,b(f_n)\}\subseteq T$ satisfies conditions MP and R1 under $\varphi$, $T_{n,j}$ also satisfies conditions MP and R1 by Lemma~\ref{bclear}.
Now if $T_{n,j}$ is closed, then $V(T)\subseteq V(T_{n,j})$ and we let $T_{n,j}=T'$. If $T_{n,j}$ is not closed, we will continue to build $T_{n,j+1}$. Eventually we will obtain a closed $T'$ as desired. 

\qed

\subsection{Proof of statement A}

\begin{proof}
	
	Let $T$ be an ETT satisfying conditions MP,  R1 and R2  with the following refinery. 
	\[
	T_1\subset  T_2 \subset 	\dots \subset T_n=T_{n,0} \subset T_{n,1}\subset 	\dots \subset T_{n,q}\subset T=T_{n,q+1}.
	\]
	We prove statement A by induction on $q$ which is the number of splitters.  We will prove the basis step and the inductive step together in the later proof. When $q=0$, we have $T_{n,q}=T_{n,0}=T_n$. Note that we have that $T_{n,0}=T_n$ is elementary because $m(T_n)=n-1<n$.  Therefore we can assume $T_{n,q}$ is elementary and show $T=T_{n,q+1}$ is also elementary. Denote $T$ by $\{T_{n,q},e_0,y_0,e_{1},...,e_p,y_p\}$ following the order $\prec_{\ell}$. We define the path number $p(T)$ of $T$ as the smallest index $i\in\{0,1,...,p\}$ such that the sequence $y_iT:=(y_i,e_{i+1},...,e_p,y_p)$ is a path in $G$.
	Suppose on the contrary that  $T$ is a counterexample to statement A, i.e., $T$ is an ETT satisfying conditions MP R1 and R2 but is not elementary.  Furthermore, we assume that among all counterexamples under $T_{n,q}$-stable colorings when $q\geq 1$ and under $T_n\cup\{f_n,b(f_n)\}$-stable coloring when $q=0$, the following two conditions hold:
	\begin{itemize}
		\item[$(1)$]   $p(T)$ is minimum,
		\item[$(2)$]   $|T-T_{n,q}|$ is minimum subject to $(1)$.
	\end{itemize}
	In the rest of this paper, when we say a coloring is $T_{n,q}$-stable in the proof, we always mean that it is $T_n\cup\{f_n,b(f_n)\}$-stable coloring when $q=0$ for convenience. 
	By our choice, $T_{y_{p-1}}$ is elementary, where $T_{y_{p-1}}=T_{n,q}$ when $p=0$. Since $T$ is not elementary, there exists a color $\alpha\in\phibar(y_p)\cap\phibar(v)$ for some $v\in V(T_{y_{p-1}})$. 
	For simplification of notations, we let $\Gamma^{q}_h=\{\gamma_{h1},\gamma_{h2}\}$ for $\delta_h\in D_{n,q}$. 
	
	\subsubsection{A few properties}
	We first introduce a few concepts that will be used later in the proof. Let $(G, e,\phiv)$ be a $k$-triple, $H$ be an elementary subgraph of $G$ and $P$ be a nonempty sub-chain of an $(\alpha,\beta)$-chain.  If $P$ is a path, $V(P)\cap V(H)=\{u\}$ is an end-vertex  and the other end-vertex of $P$ outside of $H$ has either $\alpha$ or $\beta$ as a missing color, we call the path $P$  an {\it$(\alpha,\beta)$-leg} of $H$ and $u$ an $(\alpha, \beta)$-exit. We denote the $(\alpha,\beta)$-leg $P=P^{ex}_u(\alpha,\beta,\varphi)$, where $u$ is the unique vertex in $V(P)\cap V(H)$. Two colors $\alpha$ and $\beta$ are {\it interchangeable} in $H$ if $H$ has at most one $(\alpha,\beta)$-leg. 
	
	\begin{CLA}\label{bstep1}
		For any $T_{n,j}$ with $0\leq j\leq q$ and two colors $\alpha,\beta$, if $\alpha\in\phibar_v(T_{n,j})$ and is closed in $T_{n,j}$, then $\alpha$ and $\beta$ are interchangeable in $T_{n,j}$.
		
	\end{CLA}	
	\begin{proof}
		We prove Claim~\ref{bstep1} by induction on $j$. First we consider the case when $j=0$. Since $T_n$ is closed, there is no $(\alpha,\beta)$-leg if $\beta\in\phibar_v(T_{n})$.  Hence we assume $\beta\notin\phibar_v(T_{n})$. 
		Since $T_{n,0}$ is elementary and closed, $|\partial_{\beta}(T_{n,j})|$ is odd. Hence $T_{n,j}$ has odd number of $(\alpha,\beta)$-legs. If $|\partial_{\beta}(T_{n,j})|=1$, we are done. Therefore we assume that $|\partial_{\beta}(T_{n,j})|\geq 3$. Let $u,v,w$ be three exits of $(\alpha,\beta)$-legs with $u\prec_l v\prec_l w$. Let $n'$ be the smallest index such that $w\in T_{n'}$. Then	$w\in T_{n'}-T_{n'-1}$.
		
		Let $\gamma\in\phibar(w)$. Note $T_{n}$ is closed for $\gamma$. By Lemma~\ref{bLEM:Stable}, $\varphi^*=\varphi/(G-T_{n},\alpha,\gamma)$ is $T_{n}$-stable, and $T_{n'}$ still satisfies conditions MP and R1 under $\varphi^*$. Moreover, we have that $P^{ex}_w(\gamma,\beta,\varphi^*)=P_w(\gamma,\beta,\varphi^*)=P^{ex}_w(\alpha,\beta,\varphi)$, $P^{ex}_u(\gamma,\beta,\varphi^*)=P^{ex}_u(\alpha,\beta,\varphi)$ and $P^{ex}_v(\gamma,\beta,\varphi^*)=P^{ex}_v(\alpha,\beta,\varphi)$ are three $(\gamma,\beta)$-legs of $T_{n'}$. Let the $w_2$, $u_2$ and $v_2$ be the other end vertices of $P^{ex}_w(\gamma,\beta,\varphi^*),P^{ex}_u(\gamma,\beta,\varphi^*)$ and $P^{ex}_v(\gamma,\beta,\varphi^*)$ not in $T_{n'}$. Let $u'$ be the vertex in $P^{ex}_u(\gamma,\beta,\varphi^*)$ next to $u$, and $f_u$ be the edge connecting $u$ and $u'$ in $P^{ex}_u(\gamma,\beta,\varphi^*)$; and $v'$ be the vertex in $P^{ex}_v(\gamma,\beta,\varphi^*)$ next to $v$, and $f_v$ be the edge  connecting $v$ and $v'$ in $P^{ex}_v(\gamma,\beta,\varphi^*)$. Note that $\varphi^*(f_u)=\varphi^*(f_v)=\beta$. Let $\varphi^2=\varphi^*/P^{ex}_w(\gamma,\beta,\varphi^*)$.
		Since $w\in T_{n'}-T_{n'-1}$ and $P_w(\gamma,\beta,\varphi^*)\cap T_{n'}=w$,  $T_w$ satisfies conditions MP and R1 under $\varphi^2$ by Lemma~\ref{bLEM:Stable}. Note that under $\varphi^2$, $\beta\in\phibar^2(w)$.
		Moreover, $\{T_w,f_u,u',f_v,v'\}$ is an ETT satisfying conditions MP and R1 by Lemma~\ref{bclear} because it is obtained from $T_w$ by TAA under $\varphi^2$. Applying TAA to $\{T_w,f_u,u',f_v,v'\}$ to keep adding edges and vertices until we cannot, we obtain a closed ETT $T_{n'}^2$.  Clearly, $T_{n'}^2$ satisfies conditions MP and R1 by Lemma~\ref{bclear}. Since $T_{n'}^2$ has $n'-1<n$ rungs, $T^2_{n'}$ is elementary. If one of $w_2,u_2,v_2$ is in $T^2_{n'}$, then $\gamma$ must be missing at that vertex since $\beta\in\phibar^2(T^2_{n'})$. Thus both $\gamma,\beta\in\phibar^2(T^2_{n'})$, which in turn shows that all three vertices $w_2,u_2,v_2$ are in $T^2_{n'}$. However, all of them miss either $\gamma$ or $\beta$ under $\varphi^2$ which contradicts the elementary property. 
		Thus none of these three vertices are in $T^2_{n'}$.  Hence each of $P^{ex}_u(\gamma,\beta,\varphi^*),P^{ex}_v(\gamma,\beta,\varphi^*)$ and $P^{ex}_w(\gamma,\beta,\varphi^*)$ contains a $(\gamma,\beta)$-leg of $T^2_{n'}$ under $\varphi^2$. Let $u_1,v_1,w_1$ be the corresponding exits for the $(\gamma,\beta)$-legs contained in above paths respectively. We assume without of generality, $u_1\prec_{\ell} v_1\prec_{\ell} w_1$. We have $w_1\neq w$ since we already have $w\prec_f u'\prec_f v'$ in $T^2_{n'}$. Note that $P^{ex}_{u_1}(\gamma,\beta,\varphi^2)$ and $P^{ex}_{v_1}(\gamma,\beta,\varphi^2)$ are sub-paths of $P^{ex}_u(\alpha,\beta,\varphi)$ and $P^{ex}_v(\alpha,\beta,\varphi)$ and are shorter than those two. Moreover, since $w_1\in T^2_{n'}-T_{n'-1}$, we can continue the proof process again for $T^2_{n'}$ inductively as we did for $T_{n'}$, and finally we will reach a contradiction to the elementary property because we will obtain shorter and shorter legs and finally contain all the ends.
		
		Now we suppose $j>0$ and consider the following two cases. Note that two cases bellow are similar but differ by the number of $(\alpha,\beta)$-legs.
		{\flushleft \bf Case I: $\beta\in\phibar_v(T_{n,j})$.}

		Since $T_{n,j}$ is $\alpha$-closed and, by the induction hypothesis $T_{n,j}$ is elementary, $|V(T_{n,j})|$ is odd. Therefore $|\partial_{\beta}(T_{n,j})|$ is even and there are even number of $(\alpha,\beta)$-legs. If there are none, we are done. Hence we assume that there exist two exit vertices $u,v\in T_{n,j}$, and they belong to legs $P^{ex}_u(\alpha,\beta,\varphi)$ and $P^{ex}_v(\alpha,\beta,\varphi)$, respectively. We may assume $u\preceq_l v$.
		{\flushleft \bf Case I.a: $v\in T_{n,j}-T_{n,j-1}$.	}

		Since $\beta\in\varphi(\partial(T_{n,j}))$ and $\beta\in\phibar_v(T_{n,j})$, i.e $T_{n,j}$ is not closed for $\beta$, we have $v(\beta)\in V(T_{n,j-1})$ by condition R2 (4).
		Let $\gamma\in\phibar(v)$. Then $\gamma\notin\Gamma^{j-1}$ hence $\gamma$ is closed in $T_{n,j}$ by R2. Therefore $T_{n,j}$ is closed for both $\alpha$ and $\gamma$. Hence $\varphi^*=\varphi/(G-T_{n,j},\alpha,\gamma)$ is $T_{n,j}$-stable, and conditions MP, R1 and R2 are still satisfied for $T_{n,j}$ under $\varphi^*$ by Lemma~\ref{bLEM:Stable}. However under $\varphi^*$, $P^{ex}_v(\alpha,\beta,\varphi)=P^{ex}_v(\gamma,\beta,\varphi^*)=P_v(\gamma,\beta,\varphi^*)$ and $P^{ex}_u(\alpha,\beta,\varphi)=P^{ex}_u(\gamma,\beta,\varphi^*)$ are two $(\gamma,\beta)$-legs. Let $\varphi^2=\varphi^*/P^{ex}_v(\gamma,\beta,\varphi^*)$. Then because $P^{ex}_v(\gamma,\beta,\varphi^*)\cap T_v=\{v\}$, $\varphi^2$ is $T_{v}$-stable and $T_{v}$ satisfies conditions MP, R1 and R2 by Lemma~\ref{bLEM:Stable}. 
		However, we have
		$\beta\in\phibar^2(T_{n,j-1})$ and $\beta\in\phibar^2(v)$, where we reach a contradiction to the elementary property of $T_{n,j-1}$ which comes from the minimality of $q$.
		{\flushleft \bf Case I.b: $v\in T_{n,j-1}$.	}

		We claim that there exists a color $\alpha^*\in\phibar_v(T_{n,j-1})$ that is closed in both $T_{n,j-1}$ and $T_{n,j}$.
		First we consider the case when $j=1$. Note that by condition R2(2), $|\cup_{\delta_h\in D_{n,1}}\Gamma^{0}_h|=2|D_{n,1}|\leq 2n$. Since $|\phibar_v(T_1)|\geq 13$ and $T_n$ is elementary with $|T_i|$ being odd for all $i\leq n$, we have $|\phibar_v(T_n)|\geq 11+2n\geq |\cup_{\delta_h\in D_{n,1}}\Gamma^{0}_h|$. By condition R2, $\Gamma^{0}_h\subseteq\phibar_v(T_n)$ for each $\delta_h\in D_{n,0}$. Hence we have $\Gamma^{0}_h\subseteq\phibar_v(T_n)$ for each $\delta_h\in D_{n,1}$ because $D_{n,1}\subseteq D_{n,0}$. Therefore there exists $\alpha^*\in\phibar_v(T_n)-(\cup_{\delta_h\in D_{n,1}}\Gamma^{0}_h)$. Since $T_{n,1}$ is $(\cup_{\delta_h\in D_{n,1}}\Gamma^{0}_h)^-$-closed by condition R2(4) and $T_{n}$ is closed, $\alpha^*$ is closed in both $T_{n,1}$ and $T_{n}$. Now we assume $j>1$. By condition R2(4), $T_{j-1}$ is $(\cup_{\delta_h\in D_{n,j-1}}\Gamma^{j-2}_h)^-$-closed. Similarly as the case $j=1$, we have $|\phibar_v(T_{n,j-2})|\geq 11+2n\geq |\cup_{\delta_h\in D_{n,j-1}}\Gamma^{j-2}_h|$, and there exists $\alpha^*\in\phibar_v(T_{n,j-2})-(\cup_{\delta_h\in D_{n,j-1}}\Gamma^{j-2}_h)$. Hence $\alpha^*$ is closed in $T_{n,j-1}$. By condition R2(3), $\Gamma^{j}-\Gamma^{j-1}\subseteq\phibar_v(T_{n,j}-T_{n,j-1})$, $\alpha^*\notin\Gamma^{j}$. Therefore $\alpha^*\notin (\cup_{\delta_h\in D_{n,j}}\Gamma^{j-1}_h)\subseteq\Gamma^{j}$. Now by condition R2(4), $\alpha^*$ is also closed in $T_{n,j}$, where we have the color $\alpha^*$ as claimed.
		
		Since $\alpha$ is closed in $T_{n,j}$, $\varphi^*=\varphi/(\alpha,\alpha^*,G-T_{n,j})$ is $T_{n,j}$-stable, and $T_{n,j}$ satisfies conditions MP, R1 and R2 under $\varphi^*$ by Lemma~\ref{bLEM:Stable}. Note that $\alpha^*\in\phibar^*_v(T_{n,j-1})$ and $\alpha^*$ is still closed in $T_{n,j-1}$ under $\varphi^*$.  However
		$P^{ex}_u(\alpha^*,\beta,\varphi^*)=P^{ex}_u(\alpha,\beta,\varphi)$ and $P^{ex}_v(\alpha^*,\beta,\varphi^*)=P^{ex}_v(\alpha,\beta,\varphi)$
		are two $(\alpha^*,\beta)$-legs of $T_{n,j-1}$  under $\varphi^*$, giving a contradiction to the induction hypothesis of the minimality of $j$.
		{\flushleft \bf Case II: $\beta\notin\phibar_v(T_{n,j})$}.
		
		In this case $|\partial_{\beta}(T_{n,j})|$ is odd and $\beta\notin\Gamma^{j-1}$. Hence $T_{n,j}$ has odd number of $(\alpha,\beta)$-legs. Let $u,v,w$ be exits from three $(\alpha,\beta)$-legs for $T_{n,j}$ with $u\prec_l v\prec_l w$.
		{\flushleft \bf Case II.a: 	$w\in T_{n,j}-T_{n,j-1}$.}
		
		Let $\gamma\in\phibar(w)$. By definition, $\gamma\notin\Gamma^{j-1}$. Hence $T_{n,j}$ is closed for $\gamma$ by condition R2(4). By Lemma~\ref{bLEM:Stable}, $\varphi^*=\varphi/(G-T_{n,j},\alpha,\gamma)$ is $T_{n,j}$-stable, and conditions MP, R1 and R2 are still satisfied for $T_{n,j}$ under $\varphi^*$. Moreover, under $\varphi^*$, we have $P^{ex}_w(\gamma,\beta,\varphi^*)=P_w(\gamma,\beta,\varphi^*)=P^{ex}_w(\alpha,\beta,\varphi)$, $P^{ex}_u(\gamma,\beta,\varphi^*)=P^{ex}_u(\alpha,\beta,\varphi)$ and $P^{ex}_v(\gamma,\beta,\varphi^*)=P^{ex}_v(\alpha,\beta,\varphi)$ are three $(\gamma,\beta)$-legs for $T_{n,j}$ . Let the $3$ other end vertices of $P^{ex}_w(\gamma,\beta,\varphi^*),P^{ex}_u(\gamma,\beta,\varphi^*)$ and $P^{ex}_v(\gamma,\beta,\varphi^*)$ not in $T_{n,j}$ be $w_2$, $u_2$ and $v_2$ respectively. Let $u'$ be the vertex in $P^{ex}_u(\gamma,\beta,\varphi^*)$ next to $u$, and $f_u$ be the edge connecting $u$ and $u'$ in $P^{ex}_u(\gamma,\beta,\varphi^*)$; and $v'$ be the vertex in $P^{ex}_v(\gamma,\beta,\varphi^*)$ next to $v$, and $f_v$ be the edge connecting $v$ and $v'$ in $P^{ex}_v(\gamma,\beta,\varphi^*)$. Note that $f_v$ and $f_u$ are colored $\beta$ under $\varphi^*$. Let $\varphi^2=\varphi^*/P^{ex}_w(\gamma,\beta,\varphi^*)$.
		Since $w\in T_{n,j}-T_{n,j-1}$ and $P_w(\gamma,\beta,\varphi^*)\cap T_{n,j}=w$, $T_w$ satisfies conditions MP, R1 and R2 by Lemma~\ref{bLEM:Stable}. Note that under $\varphi^2$, $\beta\in\phibar^2(w)$.
		Since $\beta\notin\Gamma^{j-1}$, we have that $\{T_w,f_u,u',f_v,v'\}$ satisfies condition R2. Moreover, it satisfies conditions MP and R1 by Lemma~\ref{bclear}. Based on the definition of condition R2, by Lemma~\ref{bclear} we can keep conditions R1 and MP by keeping extending $\{T_w,f_u,u',f_v,v'\}$ using TAA under condition R2 without using any connecting edges until it is $(\cup_{\delta_h\in D_{n,j}}\Gamma^{j-1}_h)^-$-closed. Let the resulting ETT be $T^2_{n,j}$. Clearly $T^2_{n,j}$ satisfies conditions MP, R1 and R2. By the minimality of $q$, $T^2_{n,j}$ is elementary. If one of $w_2,u_2,v_2$ is in $T^2_{n,j}$, then $\gamma$ must be missing at that vertex since $\beta\in\phibar^2(T^2_{n,j})$. Since both $\gamma,\beta\notin\Gamma^{j-1}$, and both $\gamma,\beta\in\phibar^2(T^2_{n,j})$, we must have all three vertices $w_2,u_2,v_2$ in $T^2_{n,j}$. However, all of them miss either $\gamma$ or $\beta$ under $\varphi^2$, which gives a contradiction to the elementary property. 
		Thus none of the vertices above are in $T^2_{n,j}$. Hence each of $P^{ex}_u(\gamma,\beta,\varphi^*),P^{ex}_v(\gamma,\beta,\varphi^*)$ and $P^{ex}_w(\gamma,\beta,\varphi^*)$ contains a $(\gamma,\beta)$-leg of $T^2_{n,j}$. Let $u_1,v_1,w_1$ be the exits for the $(\gamma,\beta)$-legs contained in the three paths above respectively. We without loss of generality assume $u_1\prec_f v_1\prec_f w_1$. Note that $w_1\neq w$ since we already have $w\prec_f u'\prec_f v'$ in $T^2_{n,j}$. Note that $P^{ex}_{u_1}(\gamma,\beta,\varphi^2)$ and $P^{ex}_{v_1}(\gamma,\beta,\varphi^2)$ are sub-paths of $P^{ex}_u(\alpha,\beta,\varphi)$ and $P^{ex}_v(\alpha,\beta,\varphi)$ and are shorter than those two. Moreover, since $w_1\in T^2_{n,j}-T_{n,j-1}$, we can continue the proof process again for $T^2_{n,j}$ inductively as we did for $T_{n,j}$. Continue in this fashion, we will reach a contradiction because we will obtain shorter and shorter legs until finally all the ends are contained.
		{\flushleft \bf Case II.b: 	$w\notin T_{n,j}-T_{n,j-1}$.}
		
		The proof of this case is essentially the same as in Case I.b. We first show there exists a color which is closed in both $T_{n,j-1}$ and $T_{n,j}$. So there is a $T_{n,j}$-stable coloring $\varphi^*$ in which $T_{n,j}$ satisfies all conditions MP, R1 and R2. However under $\varphi^*$, $\alpha^*$ and $\beta$ are not interchangeable in $T_{n,j-1}$, giving a contradiction to the minimality of $j$. Here we omit the proof.
	\end{proof}

	\begin{CLA}\label{9n}
		For any $y\in V(T_{y_{p-1}})-V(T_{n,q})$, $|\phibar_v(T_y)-\varphi_e(T_y - T_{n,q})|\geq11+2n$.
		Furthermore, if $|\phibar_v(T_y)- \Gamma^{q}\cup D_{n,q} \cup\varphi_e(T_y -  T_{n,q})| \leq 4$, then there exist $7$ distinct connecting colors $\delta_{i}$ with $\delta_{i}\in\phibar_v(T_{y})$ such that all colors $\delta_{i},\gamma_{i1},\gamma_{i2}\notin\varphi_e(T_y - T_{n,q})$.
	\end{CLA}
	
	\begin{proof}  Since $|\phibar (T_y -T_{n,q})| \ge |\varphi_e(T_y - T_{n,q})|$, $|V(T_{n,q})|\geq11+2(n-1)$ and
		$ |\phibar_v(T_y) - \varphi_e(T_y -T_{n,q}) | \ge |\phibar_v(T_{n,q})|  \ge |V(T_{n,q})| +2 \geq 11+2n$.
		Now assume  $|\phibar_v(T_y)- \Gamma^{q} \cup D_{n,q} \cup\varphi_e(T_y - T_{n,q})| \leq 4$.
		Since $\phibar_v(T_y) = (\phibar_v(T_y) -  \Gamma^{q}\cup D_{n,q} \cup \varphi_e(T_y - T_{n,q})) \cup
		((\Gamma^{q}\cup D_{n,q})\cap\phibar_v(T_y) - \varphi_e(T_y - T_{n,q})\cap\phibar_v(T_y)) \cup (\varphi_e(T_y-T_{n,q})\cap\phibar_v(T_y))$,
		we have
		\[
		|(\Gamma^{q}\cup D_{n,q})\cap\phibar_v(T_y) - \varphi_e(T_y - T_{n,q})\cap\phibar_v(T_y)| \ge |\phibar_v(T_y)| -  4 - |\varphi (T_y - T_{n,q})| \ge 2n + 7.
		\]
		Thus we have
		\[
		|(\Gamma^{q}\cup D_{n,q})\cap\phibar_v(T_y) - \varphi_e(T_y - T_{n,q})| \ge 2n+7.
		\]
		By the Pigeonhole Principle, there are $7$ distinguished $i$ such that  $\delta_i, \gamma_{i1}, \gamma_{i2} \notin \varphi_e(T_y - T_{n,q})$, so the result holds.  	
	\end{proof}
	
	\begin{CLA}\label{bchange}
		Let $\alpha$ and $\beta$ be two missing colors of $V(T_{y_{p-1}})$  with
		$v(\alpha) \prec_{\ell} v(\beta)$.  If $\alpha\notin\varphi_e(T_{v(\beta)} -  T_{n,q})$,
		then $P_{v(\alpha)}(\alpha,\beta,\varphi)=P_{v(\beta)}(\alpha,\beta,\varphi)$ unless at least one of $\alpha,\beta$ is in $D_{n,q}$ and $\alpha\notin\phibar(T_{n,q})$. Additionally, if $\alpha\in \phibar_v(T_{n,q})$ and $\alpha$ is $T_{n,q}$-closed, then
		$P_{v(\alpha)}(\alpha,\beta,\varphi)$ is the only $(\alpha,\beta)$-{\bf path} that may intersect $\partial(T_{n,q})$.
	\end{CLA}
	
	Note that $T_{v(\beta)}-T_{n,q}=\emptyset$ if $v(\beta)\in T_{n,q}$ and in Claim~\ref{bchange},  $(\alpha, \beta)$-path can not be replaced by  $(\alpha, \beta)$-chain because there may be  $(\alpha, \beta)$-cycles intersecting $\partial(T_m)$.
	
	\begin{proof}
		Let $v(\alpha)=u$ and $v(\beta)=w$. We consider the following few cases.
		{\flushleft \bf Case I: 	$u,w\in T_{n,q}$. }

		If $T_{n,q}$ is closed for both $\alpha,\beta$, then $E_{\alpha, \beta}\cap  \partial(T_{n,q}) =\emptyset$ and $P_{u}(\alpha,\beta,\varphi)=P_{w}(\alpha,\beta,\varphi)$ since $T_{n,q}$ is elementary.  So Claim~\ref{bchange} holds. Suppose $T_{n,q}$ is closed for $\alpha$ or $\beta$ but not for both. By Claim~\ref{bstep1} there is at most one $(\alpha,\beta)$-leg in $T_{n,q}$. If $P_{u}(\alpha,\beta,\varphi)\neq P_{w}(\alpha,\beta,\varphi)$, then both paths have one endvertex outside $T_{n,q}$, and therefore there are two $(\alpha,\beta)$-legs in $T_{n,q}$, giving a contradiction to Claim~\ref{bstep1}. If $P_{u}(\alpha,\beta,\varphi)$ is not the unique $(\alpha,\beta)$-path intersecting $\partial(T_{n,q})$, then $T_{n,q}$ has at least two $(\alpha,\beta)$-legs, where we also have a contradiction. Hence $P_{u}(\alpha,\beta,\varphi)$ is the unique $(\alpha,\beta)$-path intersecting $\partial(T_{n,q})$ and  $P_{u}(\alpha,\beta,\varphi)=P_{w}(\alpha,\beta,\varphi)$, where we have as desired. We now assume neither $\alpha$ nor $\beta$ is $T_{n,q}$-closed. Under this assumption, we only need to show that $P_{u}(\alpha,\beta,\varphi)=P_{w}(\alpha,\beta,\varphi)$. We may assume $\beta\in\phibar_v(T_{n,j'}-T_{n,j'-1})$ for some $0\leq j'<q$ where $T_{n,-1}=\emptyset$ for convenience. By condition R2, $\beta$ is closed in $T_{n,j'}$. In the same fashion as we did the case in which $T_{n,q}$ is closed for either $\alpha$ or $\beta$, we have $P_{u}(\alpha,\beta,\varphi)=P_{w}(\alpha,\beta,\varphi)$ in $T_{n,j'}$ because we have $u,v\in T_{n,j'}$.

		{\flushleft {\bf Case II: $u\in T_{n,q}$} and $w\notin T_{n,q}$. }
		
		In this case $\alpha\notin\varphi_e(T_w - T_{n,q})$.   We first consider the case that $\alpha$ is closed in $T_{n,q}$. By Claim~\ref{bstep1}, $T_{n,q}$ has at most one $(\alpha,\beta)$-leg. We also note $P_{u}(\alpha,\beta,\varphi)$ contains at least one $(\alpha,\beta)$-leg, so it is the only $(\alpha,\beta)$-leg intersecting $T_{n,q}$. If $P_{u}(\alpha,\beta,\varphi)\neq P_{w}(\alpha,\beta,\varphi)$, then $P_{w}(\alpha,\beta,\varphi)$ does not intersect $T_{n,q}$. Therefore $\varphi^*=\varphi/P_{w}(\alpha,\beta,\varphi)$ is $T_{n,q}$-stable and $T_{n,q}$ satisfies conditions MP R1 and R2 by Lemma~\ref{bLEM:Stable}. Moreover since $P_{w}(\alpha,\beta,\varphi)$ does not intersect $T_{n,q}$, all colors closed in $T_{n,s}$ under $\varphi$ stay closed under $\varphi^*$ for each $0\leq s\leq q$.  Since $\alpha,\beta\notin\varphi_e(T_{w} -  T_{n,q})-f_n$ where $f_n$ is the last connecting edge, $\varphi^*(f)=\varphi(f)$ for each $f\in E(T_w-T_{n,q})$. Hence by Lemma~\ref{bclear}, $T_w$ satisfies conditions MP and R1. Moreover, since $\varphi^*(f)=\varphi(f)$ for each $f\in E(T_w-T_{n,q})$ and all colors closed in $T_{n,s}$ under $\varphi$ stay closed under $\varphi^*$ for each $0\leq s\leq q$, $T_w$ still satisfies R2 under $\varphi^*$. Since $T_w\subset T$, $p(T_w)\leq p(T)$. However, $T_{w}$ is not elementary under $\varphi^*$, which gives a contradiction to the minimality of $p$ if $p(T_w)= p(T)$, or to the minimality of $p(T)$ if $p(T_w)<p(T)$.
		
		Now we assume that $\alpha$ is not closed in $T_{n,q}$. In this case we only need to prove $P_{u}(\alpha,\beta,\varphi)=P_{w}(\alpha,\beta,\varphi)$. Suppose $P_{u}(\alpha,\beta,\varphi)\neq P_{w}(\alpha,\beta,\varphi)$. Recall that $\alpha\in\phibar_v(T_{n,q})$. If $\alpha\in\phibar_v(T_n)$, it is closed in $T_{n}=T_{n,0}$. If $\alpha\in\phibar_v(T_{n,j}-T_{n,j-1})$ for some $1\leq j\leq q$, then $\alpha\notin\Gamma^{j-1}$, and therefore $\alpha$ is closed in $T_{n,j}$ by condition R2(4). So either way, there exists the largest $q'$ such that $\alpha$ is closed in $T_{n,q'}$.
		Since the only edge in $T-T_n$ with color not missing before is the connecting edge with color $\delta_n$, we have $\beta\notin \phiv_e(T_w-T_{n}-f_n)$. We claim that $\alpha\notin \phiv_e(T_w-T_{n,q'})$. Suppose $\alpha\in\phiv_e(T_w-T_{n,q'})$. We can assume $\alpha\in\phiv_e(T_{n,r}-T_{n,r-1})$ for some $r>q'$. Then $\alpha\notin\cup_{ \delta_h\in D_{n,r}}\Gamma^{r-1}_h$ by condition R2(1), and therefore $\alpha$ is closed in $T_{n,r}$ by condition R2(4), which contradicts the maximality of $q'$. Hence we indeed have $\alpha\notin\phiv_e(T_w-T_{n,q'})$. By Claim~\ref{bstep1} there is at most one $(\alpha,\beta)$-leg in $T_{n,q'}$, which is a sub-path of $P_{u}(\alpha,\beta,\varphi)$. Then $P_{w}(\alpha,\beta,\varphi)$ is disjoint with $T_{n,q'}$. Hence under $\varphi^*=\varphi/P_{w}(\alpha,\beta,\varphi)$, $T_{n,q'}$ satisfies conditions MP, R1 and R2 by Lemma~\ref{bLEM:Stable}. Moreover, $T_w$ satisfies conditions MP and R1 by Lemma~\ref{bclear} since $\alpha\notin\phiv_e(T_w-T_{n,q'})$ and $\beta\notin\phiv_e(T_w-T_n-f_n)$. Now we first consider the case when $q'\geq 1$ or $\beta\neq\delta_n$. In this case $\alpha,\beta\notin\phiv_e(T_w-T_{n,q'})$. Therefore, $\alpha,\beta\notin\phiv_e^*(T_w-T_{n,q'})$, and $\varphi^*$ is $T_{n,q}$-stable, which implies $T_w$ and all $T_{n,s}$ for $q'\leq s\leq q$ satisfies conditions MP, R1, and R2(1), (2) (3) under $\varphi^*$. Note that neither $\alpha$ nor $\beta$ is closed in $T_{n,s}$ for $q'\leq s\leq q$ for $\varphi$, we have $T_w$ and all $T_{n,s}$ for $q'\leq s\leq q$ satisfies condition R2(4) because none of the closed colors become non-closed in $T_{n,s}$ for $q'\leq s\leq q$ under $\varphi^*$. However, $\alpha\in\varphi^*(w)\cap\varphi^*_e(T_{n,q})$, giving a contradiction to the minimality of $p$ if $p(T_w)= p(T)$, or to the minimality of $p(T)$ if $p(T_w)<p(T)$. For the case $q'=0$ and $\beta=\delta_n$, we have that $\beta$ is only assigned to the connecting edge $f_n$ in $T_{n,1}-T_n$  by the construction of $T_{n,1}$ and the assumption of $\beta\in\phibar(w)$ and $w\notin T_{n,q}$. Moreover, by Claim~\ref{bstep1}, $\alpha$ is interchangeable with $\beta$ in $T_{n,0}$, hence there is only one $(\alpha,\beta)$-leg in $T_{n,0}$. Therefore $P_{w}(\alpha,\beta,\varphi)$ is disjoint with $T_{n,0}$, and therefore $\varphi^*(f_n)=\beta$. Note that we can conclude $\varphi^*$ is $T_w$-stable as before, so similarly we have $T_w$ satisfying conditions MP, R1 and R2. We then reach a contradiction since $\alpha\in\phibar^*(w)\cap\phibar^*_v(T_{n,q})$.

		{\flushleft \bf Case III: 	$u,w\notin T_{n,q}$.}
		
		In this case,  we have  $\alpha\notin\varphi_e(T_w -  T_{n,q})$ and $\alpha,\beta \notin D_{n,q}$, which in turn give $\alpha, \beta \notin D_{n}\cup \phibar(T_{w}-w)$. Then $\alpha,\beta\notin\varphi_e(T_w)$.   Suppose on the contrary that $P_{u}(\alpha,\beta,\varphi)\neq P_{w}(\alpha,\beta,\varphi)$. Now consider the proper coloring $\varphi^*=\varphi/P_{w}(\alpha,\beta,\varphi)$. Since  $\alpha, \beta \notin D_{n}\cup \phibar(T_{w}-w)$, all the edges colored by connecting colors and their companion colors stay the same under $\varphi^*$ as $\varphi$ and therefore $T_w$ is an ETT under $\varphi^*$. Moreover, each $T_i$ is still closed under $\varphi^*$ for $1\leq i\leq n$. Since $\alpha\notin\varphi_e(T_{w} -  T_{n,q})$,   $\varphi(f)=\varphi^*(f)$ for every $f\in E(T_w)$ and $\phibar(v)=\phibar^*(v)$ for every $v\in T_w-w$. Thus $T$ is still an ETT under $\varphi^*$. Therefore $\varphi^*$ is $T$-stable and $T$ satisfies conditions MP and R1 by Lemma~\ref{MPR1}. Moreover,  
		$T$ still satisfies condition R2 under $\varphi^*$ because R2 is not related to colors in $\phibar_v(T-T_{n,q})-D_{n,q}$. However, now $\alpha\in\phibar^*(u)\cap\phibar^*(w)$, which gives a contradiction to the minimality of $|T-T_{n,q}|$.
	\end{proof}
	
	\begin{CLA}\label{bstablechange}
		For any two colors $\alpha,\beta\in\phibar_v(T_{y_{p-1}})$, the following two statements hold.
		\begin{enumerate}
			\item	If $\alpha\in\phibar_v(T_{n,q})$ and $P$ is an $(\alpha,\beta)$-path other than $P_{v(\alpha)}(\alpha,\beta,\varphi)$, then $T_{n,q}$ satisfies conditions MP, R1 and R2 under the $T_{n,q}$-stable coloring $\varphi^*=\varphi/P$. Here by condition R2 holds, we mean R2(1) holds for $j<q$, and R2(2), (3) and (4) hold for $j\leq q$.
			\item If $T_{n,q}\prec_{\ell} v(\alpha) \prec_{\ell} v(\beta) \prec_{\ell} y_{p-1}$, $\alpha\notin \phiv_e(T_{v(\beta)})$ and $\alpha, \beta \notin D_{n,q}$, then $\varphi^*=\phiv/P$ is $T$-stable for any $(\alpha, \beta)$-chain $P$. Consequently, $T$ satisfies conditions MP, R1 and R2 under $\varphi^*$.
		\end{enumerate}
	\end{CLA}
	\begin{REM}\label{beasy}
		In part (1) by condition R2 holds, we mean R2(1) holds for $j<q$, and R2(2), (3) and (4) hold for $j\leq q$. Therefore after applying Claim~\ref{bstablechange}, we only need to show that $T-T_{n,q}$ satisfies condition R2(1) for $j=q$ to confirm that $T$ also satisfies condition R2 under $\varphi^*$. Moreover, after applying Claim~\ref{bstablechange} we could always apply Lemma~\ref{bclear} to show that $T$ also satisfies conditions MP and R1 if $T$ is still an ETT obtained from $T_{n,q}$ by TAA under $\varphi^*$.
	\end{REM}
	\begin{proof}
		We first prove (1).
		If one of $\alpha$ and $\beta$ is closed in $T_{n,q}$, we have that $P_{v(\alpha)}(\alpha,\beta,\varphi)$ is the only $(\alpha,\beta)$-path intersecting $T_{n,q}$ by Claim~\ref{bchange}. Therefore $P$ is disjoint with $T_{n,q}$. Then $\varphi^*=\varphi/P$ is a $T_{n,q}$-stable coloring and $T_{n,q}$ satisfies conditions MP, R1, R2 under $\varphi^*$ by Lemma~\ref{bLEM:Stable}. We now suppose that neither $\alpha$ nor $\beta$ is closed in $T_{n,q}$. Then similarly to the proof of Claim~\ref{bchange}, by condition R2(4), there exist the largest $q'$ such that either $\alpha$ or $\beta$ is closed in $T_{n,q'}$.  First we consider the case when $\beta\neq\delta_n$ or $q'>0$. We claim that $\alpha,\beta\notin \phiv_e(T_{n,q}-T_{n,q'})$. The proof of $\alpha\notin \phiv_e(T_{n,q}-T_{n,q'})$ is the same as in Claim~\ref{bchange} Case II, where we assume $\alpha$ is not closed in $T_{n,q}$. Now we prove $\beta\notin\phiv_e(T_{n,q}-T_{n,q'})$. If $\beta\in\phibar_v(T_{n,q})$, we argue just as in the case when $\alpha$ is not closed in $T_{n,q}$. If $\beta\notin\phibar_v(T_{n,q})$, the only possibility that $\beta\in\phiv_e(T_{n,q}-T_{n,q'})$ is $\beta=\delta_n$ and $q'=1$ by the definition of ETT, which does not meet the assumption of this case. Hence $\alpha,\beta\notin \phiv_e(T_{n,q}-T_{n,q'})$.  By Claim~\ref{bstep1} there is at most one $(\alpha,\beta)$-leg in $T_{n,q'}$, which is a sub-path of $P_{v(\alpha)}(\alpha,\beta,\varphi)$. Then $P$ is disjoint with $T_{n,q'}$. Hence $T_{n,q'}$ satisfies conditions MP, R1 and R2 and $\varphi^*$ is $T_{n,q'}$-stable by Lemma~\ref{bLEM:Stable}. Since $\alpha,\beta\notin \phiv_e(T_{n,q}-T_{n,q'})$, $T_{n,q}$ satisfies (1) (2) and (3) of condition R2 under $\varphi^*$ and $\varphi^*$ is $T_{n,q}$-stable. Moreover, $T_{n,q}$ satisfies conditions MP and R1 by Lemma~\ref{MPR1}. Since both $\alpha,\beta$ are not closed in $T_{n,t}$ for $q'\leq t\leq q$, $T_{n,t}$ satisfies condition R2 (4) because none of the closed colors become non-closed in $T_{n,t}$. Thus $T_{n,q}$ itself satisfies condition R2 (4) and we have as desired. For the case $q'=0$ and $\beta=\delta_n$, the only edge $f_n$ colored by $\beta$ is a connecting edge by the construction of $T_{n,1}$, and $\beta\notin \phiv_e(T_{n,q}-T_{n,1}),\alpha\notin \phiv_e(T_{n,q}-T_{n,0})$. Moreover, by Claim~\ref{bstep1}, $\alpha$ is interchangeable with $\beta$ in $T_{n,0}$. Hence there is only one $(\alpha,\beta)$-leg in $T_{n,0}$. Therefore $P$ is disjoint with $T_{n,0}$, and $\varphi^*(f_n)=\beta$. Note that we can conclude $\varphi^*$ is $T_{n,q}$-stable as the case earlier and prove that $T_{n,q}$ satisfies conditions MP, R1 and R2, where we can proceed as earlier to finish the proof.
		
		Now we prove (2). By Claim~\ref{bchange}, $P_{v(\alpha)}(\alpha,\beta,\varphi)=P_{v(\beta)}(\alpha,\beta,\varphi)$. In this case we have $\alpha,\beta\notin\phibar_v(T_{n,q})\cup D_{n}$ and $\alpha,\beta\notin\varphi_e(T_{v(\beta)})$. If $P=P_{v(\alpha)}(\alpha,\beta,\varphi)$, then $T$ is an ETT under $\varphi^*=\varphi/P$ since $\alpha,\beta\notin\varphi^*_e(T_{v(\beta)})$ and $\alpha,\beta\notin\phibar_v(T_{n,q})\cup D_{n}$. Moreover, all the edges colored by connecting colors and their companion colors stay the same under $\varphi^*$ as $\varphi$. Thus $\varphi^*$ is $T$-stable. By Lemma~\ref{MPR1}, $T$ also satisfies condition MP and R1. Since $\alpha,\beta\in\phibar_v(T-T_{n,q})-D_{n,q}$ and condition R2 is not related to colors in $\phibar_v(T-T_{n,q})-D_{n,q}$, $T$ satisfies condition R2 under $\varphi^*$.  If $P\neq P_{v(\alpha)}(\alpha,\beta,\varphi)$, then similarly $\varphi^*$ is $T$-stable and $T$ is an ETT satisfying conditions MP and R1 under $\varphi^*$ by Lemma~\ref{MPR1} because we still have  $\alpha,\beta\notin\varphi^*_e(T_{v(\beta)})$ and $\alpha,\beta\notin\phibar_v(T_{n,q})\cup D_{n}$. Moreover, since $\alpha,\beta\in\phibar_v(T-T_{n,q})-D_{n,q}$, $T$ still satisfies condition R2 under $\varphi^*$.
	\end{proof}
	
	\subsubsection{Case verification}\label{verify}
	\begin{CLA}\label{bcla-p>0}
		$p > 0$
	\end{CLA}
	\proof
	Suppose on the contrary $p=0$, that is, $T=T_{n,q}\cup\{e_0,y_0\}$. We consider two cases.
	
	{\flushleft {\bf Case I:}  $q=0$. In this case $T_{n,q}=T_n$ is closed and $e_0$ is a connecting edge.}		
	
	Let $\alpha\in\phibar_v(T_{n,q})\cap\phibar(y_0)$. Let $k$ be the minimum index such that $\gamma_n \in \phibar_v(T_k)$. By condition R1, $\gamma_n\in\phibar_v(T_k)$ and $\gamma_n\notin\varphi_e(T_n)$.  Let $\varphi^*=\varphi/(G-T_n,\gamma_n,\alpha)$. By Lemma~\ref{bLEM:Stable}, $T$ satisfies conditions MP, R1 and R2 under the $T$-stable coloring $\phiv^*$. Note that $P_{v(\gamma_n)}(\delta_n,\gamma_n,\varphi^*)=P_{y_0}(\delta_n,\gamma_n,\varphi^*)$ by Claim~\ref{bchange} where $v(\gamma_n)$ is the unique vertex in $V(T_k)$ such that $\gamma_n \in \phibar(v(\gamma_n))$. Because $e_0$ is the first edge in $P_{v(\gamma_n)}(\delta_n,\gamma_n,\varphi)\cap\partial(T_n)$, $P_{v(\gamma_n)}(\delta_n,\gamma_n,\varphi^*)$ contains only one edge colored $\delta_n$ in $\partial(T_n)$ under the coloring $\varphi^*$. Hence there is another $(\delta_n,\gamma_n)$-chain $Q$ intersecting  $\partial(T_n)$ under $\varphi^*$. Let $s$ be the smallest index with $s\geq k$ such that $V(Q)\cap V(T_s)\neq\emptyset$. Let $\varphi^2=\varphi^*/Q$. We claim that $\varphi^2$ is a $T_s$-stable coloring. We first consider the case $s>k$. Then $V(Q)\cap V(T_{s-1})=\emptyset$. By Lemma~\ref{bLEM:Stable}, $\varphi^2$ is $T_{s-1}\cup\{f_s,b(f_s)\}$-stable and $T_{s-1}\cup\{f_s,b(f_s)\}$ is still an ETT under $\varphi^2$.  Because $\gamma_n\notin\varphi_e(T_s)$ and $\delta_n\notin\varphi_e(T_s-T_{s-1}-f_{s-1})$ where $f_{s-1}$ is a connecting edge, we have $\varphi^*(f)=\varphi^2(f)$ for all $f$ incident to $V(T_{s-1})$, $\varphi^*(f)=\varphi^2(f)$ for all $f\in E(T_{s-1})$ and $\phibar^*(v)=\phibar^2(v)$ for all $v\in V(T_s)$. Therefore $T_s$ is still an ETT under $\varphi^2$ with the same connecting colors, connecting edges and companion colors as under $\varphi^*$, and $\varphi(f)=\varphi^*(f)$ for every $f$ incident to $V(T_{s-1})$ if $\varphi(f)\in\{\delta_i,\gamma_i\}$ or $\varphi^*(f)\in\{\delta_i,\gamma_i\}$ with $1\leq i\leq s-1$. Hence $\varphi^2$ is $T$-stable. We then assume $s=k$. By condition R1 and the definition of ETT, $\delta_n$ has not been used as a connecting color in $T_s$ and $\gamma_n$ has not been used as a companion color in $T_s$, i.e., $\gamma_i\neq\gamma_n$ and $\delta_i\neq\delta_n$ for $1\leq i< s$. Therefore  we also have that $\varphi(f)=\varphi^*(f)$ for every $f$ incident to $V(T_{s-1})$ if $\varphi(f)\in\{\delta_i,\gamma_i\}$ or $\varphi^*(f)\in\{\delta_i,\gamma_i\}$ with $1\leq i\leq s-1$, and $T_s$ is an ETT under $\varphi^2$ with the same connecting colors, connecting edges and companion colors as under $\varphi^*$. Thus $\varphi^2$ is $T_s$-stable in both cases.  However, under the coloring $\phiv^2$, $T_s$ is no longer closed,    giving a contradiction to MP condition of $T$.
	
	{\flushleft {\bf Case II:}  $q>0$. In this case $T_{n,q}$ is not closed although it is $(\cup_{\delta_h\in D_{n,q}}\Gamma^{q-1}_h)^-$-closed.}	
	
	Assume without loss of generality that $e_0$ is colored by $\gamma_0\in\Gamma^{q-1}$. Let $\alpha\in\phibar_v(T_{n,q})\cap\phibar(y_0)$. Let $u=a(e_0)$ the endvertex of $e_0$ in $T_{n,q}$.  We further assume that $u\in T_{n,q'}-T_{n,q'-1}$ for some $q'\leq q$ where $T_{n,-1}=\emptyset$ for convenience. We claim that $v(\gamma_0)\prec_l u$. Otherwise we can assume $v(\gamma_0)\in T_{n,s}-T_{n,s-1}$ for some $q'\leq s\leq q$, and then $\gamma_0$ is closed in $T_{n,s}$ by condition R2(4). Combining with the assumption $u\prec_f v(\gamma_0)$, we get $y_0\in T_{n,s}$, a contradiction. Hence we have as claimed. Let $\gamma\in\phibar(u)$. Clearly $\alpha\neq\gamma_0$ and $\gamma\neq\gamma_0$ because $\varphi(e_0)=\gamma_0$ and $\alpha,\gamma$ are missing at the endvertices of $e_0$. Since both $v(\alpha)$ and $u$ are in $T_{n,q}$, we have $P_{v(\alpha)}(\alpha,\gamma,\varphi)=P_{u}(\alpha,\gamma,\varphi)$ by Claim~\ref{bchange}, and $P_{y_0}(\alpha,\beta,\varphi)$ is different from the path above. Moreover by Claim~\ref{bstablechange}, $\varphi^*=\varphi/P_{y_0}(\alpha,\gamma,\varphi)$ is $T_{n,q}$-stable and $T_{n,q}$ satisfies conditions MP, R1 and R2. Since $\alpha\neq\gamma_0$ and $\gamma\neq\gamma_0$, $\varphi^*(e_0)=\gamma_0$. Note that $\gamma\in\phibar^*(u)\cap\phibar^*(y_0)$. Let $\varphi^2$ be the coloring obtained from $\varphi^*$ by recoloring $e_0$ by $\gamma$. Hence $\gamma_0\in\phibar^2(u)$. In the case $u\notin T_{n}$, since $\varphi^*$ and $\varphi^2$ are only differed by one edge $e_0$ and both endvertices of $e_0$ are outside of $T_u-u$, $\varphi^2$ is $T_u$-stable and $T_u$ satisfies conditions MP, R1 and R2 by Lemma~\ref{bLEM:Stable}. Since $v(\gamma_0)\prec_l u$, we have $\gamma_0\in\phibar^2(u)\cap\phibar^2(v(\gamma_0))$, giving a contradiction to the minimality of $q$. Now we assume $u\in T_{n}$. Then similarly $\varphi^2$ is $T_u$-stable and $T_u$ satisfies conditions MP and R1 by Lemma~\ref{bLEM:Stable}. Since $m(T_w)\leq n-1$, $T_w$ is elementary under $\varphi^2$ by our assumption. However, we have $\gamma_0\in\phibar^2(u)\cap\phibar^2(v(\gamma_0))$, giving a contradiction.
	\qed
	\begin{case}
		\label{bcase1x} $p(T)=0$.
	\end{case} 
	In this case $T - T_{n,q}$ is a path, so we call $T$ a Generalized Kierstead path.
	\begin{CLA}\label{bclaim9}
		We may assume $\alpha\in\phibar(y_i)\cap\phibar(y_p)$ for some $0\leq i<p$.
	\end{CLA} 	
	
	\proof Suppose  $\alpha\in\phibar(y_p)\cap\phibar(v)$ for some
	$v\in V(T_{n,q})$.  We first consider the case  $\alpha\notin\varphi_e(T -  T_{n,q})$. Let $\beta\in\phibar(y_{p-1})$. By  Claim~\ref{bchange} $P_v(\alpha,\beta,\varphi)=P_{y_{p-1}}(\alpha,\beta,\varphi)$ and  $P_{y_{p}}(\alpha,\beta,\varphi)$ is difference from the path above.  Let $\varphi^* := \varphi/P_{y_p}(\alpha, \beta, \varphi)$.  By Claim~\ref{bstablechange},   $T_{n,q}$ is $\varphi/\varphi^*$-stable and satisfies conditions MP, R1 and R2. Since $\alpha,\beta\notin \varphi_e(T_{y_{p}}-T_{n,q})$, $T$ clearly satisfies condition R2 under $\varphi^*$. By Lemma~\ref{bclear}, $T$ also satisfies MP and R1 under $\varphi^*$.  Note that we have $\beta \in \phibar^*(y_{p-1})\cap \phibar^*(y_p)$,  Claim~\ref{bclaim9} holds.
	
	We now consider the case $\alpha\in\varphi_e(T - T_{n,q})$.
	Following order $\prec_{\ell}$, let $e_j$ be the first edge in $T -T_{n,q}$ such that $\alpha=\varphi(e_j)$. 
	We first consider the case $j\geq 1$. Let $\beta \in \phibar(y_{j-1})$. By Claim~\ref{bchange}, $P_v(\alpha,\beta,\varphi)=P_{y_{j-1}}(\alpha,\beta,\varphi)$ and  $P_{y_{p}}(\alpha,\beta,\varphi)$ is different from the path above. Moreover, by Claim~\ref{bstablechange} $\varphi^*=\varphi/P_{y_{p}}(\alpha,\beta,\varphi)$ is a $T_{n,q}$-stable coloring and $T_{n,q}$ satisfies conditions MP, R1 and R2. Moreover, $T$ is still an ETT obtained from $T_{n,q}$ by TAA under $\varphi^*$ and therefore it satisfies MP and R1 by Lemma~\ref{bclear}. Clearly, under $\varphi^*$, condition R2 holds for $T$ if $\alpha\notin \Gamma^{q}$.  If $\alpha\in \Gamma^{q}$, say  $\alpha=\gamma_{i1}$ for some $0<i\leq n$, by condition R2 we have $\delta_i\in\phibar(w)$ for some $w\preceq_{\ell} y_{j-1}$. Since only edges after $w$ in the order $\prec_{\ell}$ may change colors between $\alpha$ and $\beta$, condition R2 also holds under $\varphi^*$.  Since $\beta\in \phibar^*(y_{j-1})\cap \phibar^*(y_p)$,  Claim~\ref{bclaim9} holds by simply denoting $\varphi^*$ as $\varphi$. From now on when we claim that $T$ satisfies MP, R1 and R2 without checking MP and R1, we either follow Lemma~\ref{bclear}, or the checking of those conditions are trivial, especially after applying Claim~\ref{bchange}.
	
	Now we assume that $j=0$. In this case $q>0$, since $q=0$ implies $\alpha=\delta_n$ which is a contradiction to $\alpha\in\phibar_v(T_{n,q})$. Therefore, $\alpha=\varphi(e_0)$ where $\alpha\in\Gamma^{q-1}$. Note that $\alpha\notin\Gamma^{q}$ by condition R2.  We will show that there exists $ \gamma\in\phibar_v(T_{n,q})-\Gamma^{q}$ such that $\gamma$ is closed in $T_{n,q}$. By condition R2(4), $T_{n,q}$ is $(\cup_{\delta_h\in D_{n,q}}\Gamma^{q-1}_h)^-$-closed. Therefore, $T_{n,q}$ is closed for colors in $\phibar_v(T_{n,q})-\Gamma^{q-1}$ because $\cup_{\delta_h\in D_{n,q}}\Gamma^{q-1}_h\subseteq \Gamma^{q-1}$. Hence we need to show that there exists $\gamma\in\phibar_v(T_{n,q})-\Gamma^{q}\cup\Gamma^{q-1}$. Since $\Gamma^{q}-\Gamma^{q-1}\subseteq \phibar_v(T_{n,q}-T_{n,q-1})$ by condition R2 and the assumption that $T_{n,q-1}$ is elementary, we have $|(\Gamma^{q}\cup\Gamma^{q-1})\cap\phibar_v(T_{n,q-1})|=|\Gamma^{q-1}\cap\phibar_v(T_{n,q-1})|\leq 2n$ and $|\phibar_v(T_{n,q-1})|\geq |\phibar_v(T_1)|+2(n-1)=2n+11$. Therefore $|\phibar_v(T_{n,q})-\Gamma^{q}\cup\Gamma^{q-1}|=|\phibar_v(T_{n,q-1})-\Gamma^{q-1})|+|\phibar_v(T_{n,q}-T_{n,q-1})-(\Gamma^{q}-\Gamma^{q-1})|\geq|\phibar_v(T_{n,q-1})-\Gamma^{q-1}|\geq (2n+11)-2n\geq 11$, where we have $\gamma$ as desired. Now by Claim~\ref{bchange}, $P_{v(\alpha)}(\alpha,\gamma,\varphi)=P_{v(\gamma)}(\alpha,\gamma,\varphi)$, and $P_{y_p}(\alpha,\beta,\varphi)$ is disjoint with $T_{n,q}$. Therefore by Claim~\ref{bstablechange}, $T_{n,q}$ is stable under $\varphi^*=\varphi/P_{y_p}(\alpha,\gamma,\varphi)$ and satisfies conditions MP, R1 and R2. Moreover, since $e_0\in\partial(T_{n,q})$, $e_0\notin P_{y_p}(\alpha,\gamma,\varphi)$. Since $\alpha,\gamma\notin\Gamma^{q}$, $Ty_{p}$ satisfies conditions MP, R1 and R2 under $\varphi^*$. Now $\gamma\in\phibar^*(y_p)\cap\phibar^*(v)$ for some
	$v\in V(T_{n,q})$ and $\alpha\neq\gamma$, which returns to the case either $\gamma\notin\phibar_v(T-T_{n,q})$ or $j\geq 1$.
	\qed
	
	Among all $T$-stable colorings satisfying conditions MP, R1 and R2, we assume that $i$ is the maximum index such that  $\alpha\in\phibar(y_p)\cap\phibar(y_i)$.
	
	\begin{CLA}\label{bp-1}
		$i=p-1$.
	\end{CLA}
	\proof  Suppose on the contrary  $i<p-1$. We first consider the case $\alpha\notin D_{n,q}$.
	Let $\theta\in\phibar(y_{i+1})$. If $\theta\notin D_{n,q}$,  then
	$\{\alpha, \theta\}\cap D_{n,q} = \emptyset$.  By Claim~\ref{bchange}, $P_{y_{i}}(\alpha,\theta,\varphi)=P_{y_{i+1}}(\alpha,\theta,\varphi)$. Let $\varphi^*=\varphi/P_{y_{p}}(\alpha,\beta,\varphi)$. Since both $y_i,y_{i+1}\in T-T_{n,q}$ and $\alpha\notin T_{y_{i+1}}$, by Claim~\ref{bstablechange}, $T$ is also an ETT under $\varphi^*$ and conditions MP, R1 and R2 hold.   But $\theta\in\phibar^*(y_p)\cap\phibar^*(y_{i+1})$, which contradicts the maximality of $i$. We now consider the case $\theta=\delta_k$ for some $k\leq n$.
	By Claim~\ref{bchange}, $P_{v(\gamma_{k1})}(\alpha,\gamma_{k1},\varphi)=P_{y_{i}}(\alpha,\gamma_{k1},\varphi)$ and $P_{y_{p}}(\alpha,\gamma_{k1},\varphi)$ is different from path above.  Note that $P_{y_{p}}(\alpha,\gamma_{k1},\varphi)=y_p$ can occur if $\gamma_{k1}\in\phibar(y_p)$.  By Claim~\ref{bstablechange},  $\varphi^*=\varphi/P_{y_{p}}(\alpha,\gamma_{k1},\varphi)$ is a $T_{n,q}$-stable coloring and $T_{n,q}$ satisfies conditions MP, R1 and R2. Moreover, since $\alpha\notin\varphi_e(Ty_{i+1})$ and
	$\delta_k\in\phibar(y_{i+1})$, $T$ still satisfies conditions MP, R1 and R2 under $\varphi^*$. Then, by Claim~\ref{bchange} again, $P_{v(\gamma_{k1})}(\delta_k,\gamma_{k1},\varphi^*)=P_{y_{i}}(\delta_k,\gamma_{k1},\varphi^*)$ and $P_{y_{p}}(\delta_k,\gamma_{k1},\varphi^*)$ is different from the path above. Let $\varphi^{**}=\varphi^*/P_{y_{p}}(\delta_k,\gamma_{k1},\varphi^*)$.  By Claim~\ref{bstablechange},  $\varphi^{**}$ is $T_{n,q}$-stable and  $T_{n,q}$ satisfies conditions MP, R1 and R2 under $\varphi^{**}$. Note that $\gamma_{k1}\notin\phibar^{**}_v(T_{y_{i+1}}-T_{n,q})$, and $\delta_k$ is only used by connecting edges in $T_{y_{i+1}}$ where they are colored the same under $\varphi^{**}$, we see that $T$ satisfies condition R2. However under $\varphi^{**}$, $\delta_k\in\phibar^{**}(y_p)\cap\phibar^{**}(y_{i+1})$,  giving a contradiction to the maximality of $i$.

	We now consider the case $\alpha\in D_{n,q}$, say $\alpha=\delta_k$ for some $k\leq n$. Since $\varphi(e_{i+1})$ can not be both $\gamma_{k1}$ and $\gamma_{k2}$, we assume without loss of generality  $\varphi(e_{i+1})\neq\gamma_{k1}$.   By Claim~\ref{bchange}, $P_{v(\gamma_{k1})}(\delta_k,\gamma_{k1},\varphi)=P_{y_{i}}(\delta_k,\gamma_{k1},\varphi)$ and $P_{y_{p}}(\delta_k,\gamma_{k1},\varphi)$ is a different path. Let  $\varphi^*=\varphi/P_{y_{p}}(\delta_k,\gamma_{k1},\varphi)$.  By Claim~\ref{bstablechange},  $\varphi^*$ is $T_{n,q}$-stable and $T_{n,q}$ satisfies conditions MP, R1 and R2. Now $\gamma_{k1}\in\phibar^*(y_p)$. Moreover, since $\gamma_{k1}\notin\phibar_v(T_{y_i}-T_{n,q})$ and $\delta_k$ is only used by connecting edges in $T_{y_{i}}$ where they are colored the same under $\varphi^*$, $T$ satisfies conditions MP, R1, R2 and $\varphi^*(e_{i+1})\neq\gamma_{k1}$. Hence $\gamma_{k1}\notin\varphi^*_e(T_{y_{i+1}}-T_{n,q})$. Let $\theta\in\phibar^*(y_{i+1})$. By the minimality of $V(T-T_{n,q})$, $\theta\neq\gamma_{k1}$. By Claim~\ref{bchange}, $P_{v(\gamma_{k1})}(\theta,\gamma_{k1},\varphi^*)=P_{y_{i+1}}(\theta,\gamma_{k1},\varphi^*)$, and $P_{y_{p}}(\theta,\gamma_{k1},\varphi^*)$ is a different path.  By Claim~\ref{bstablechange},  $\varphi^{**}=\varphi^*/P_{y_{p}}(\theta,\gamma_{k1},\varphi^*)$ is $T_{n,q}$-stable and $T_{n,q}$ satisfies conditions MP, R1 and R2. Moreover, since $\theta\in\varphi_e(T_{y_{i+1}})$ implies $\theta=\delta_m$ ($m\leq n$) which is only used by some connecting edges in $T_{y_{i+1}}$ and $\gamma_{k1}\notin\varphi^*_e(T_{y_{i+1}}-T_{n,q})$, $\varphi^{**}$ is $T_{n,q}$-stable ensures $T$ satisfies conditions MP, R1 and R2 under $\varphi^{**}$. However, $\theta\in\phibar^{**}(y_p)\cap\phibar^{**}(y_{i+1})$,  giving a contradiction to the maximality of $i$. \qed
	
	Now we have $i=p-1$. Let $\varphi(e_{p})=\theta$. Since $\alpha\in\phibar(y_p)\cap\phibar(y_{p-1})$, we can recolor $e_{p}$ by $\alpha$. Denote the new coloring by $\varphi^*$. Then $\theta\in\phibar^*(y_{p-1})$, and $\varphi^*$ is $T_{y_{p-1}}$-stable. Moreover, $T_{y_{p-1}}$ satisfies conditions MP, R1 and R2 under $\varphi^*$. Note that $v(\theta)\prec_{\ell} y_{p-1}$, we have a counterexample which has one less vertices than $T$, giving a contradiction. \qed
	
	\begin{case}\label{b2}
		$p(T)=p\geq 1$. In this case, $y_{p-1}$ is not incident to $e_p$. Let $\theta =\varphi(e_p)$.
	\end{case}
	We divide this case into a few subcases. In summary,  we will prove Case 2.1.1 independently.  Case 2.1.2 is redirected to Case 2.1.1 and Case 2.2.1, Case 2.2.1 is redirected to Case 2.1.1, Case 2.2.2 is redirected to Case 2.2.1., which is further redirected to Case 2.1.1. Case 2.3.1 is redirected to Case 2.1, Case 2.3.2 is redirected to Case 2.1.1 and Case 2.2.  Therefore, at the end, there is no loophole in our proof. The next figure describes the above.
	\begin{figure}[!htb]
		\begin{center}

			\includegraphics[height=3.5cm]{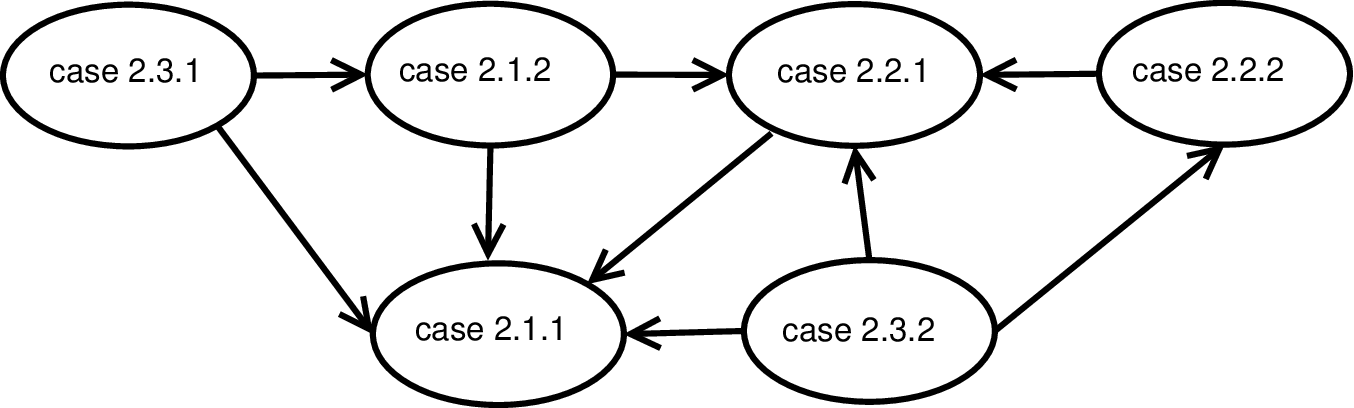}
			
		\end{center}
	\end{figure}	
	\begin{subcase}\label{bcase611}
		$\alpha\in\phibar(y_p)\cap\phibar(y_{p-1})$ and $\alpha\in D_{n,q}$. 
	\end{subcase}
	Assume $\alpha=\delta_{m}$ for some $m\leq n$.
	Since $\delta_{m}\in\phibar(y_{p})$, we have $\theta\neq\delta_{m}$. Note that $\theta\in D_{n,q}$ may occur.
	\begin{subsubcase}\label{bcase611a}
		$\theta\notin\phibar(y_{p-1})$.
	\end{subsubcase}
	We first consider the case $\theta\notin\Gamma^{q}$.
	By condition R2, $\gamma_{m1},\gamma_{m2}\notin \varphi_e(T_{y_{p-1}} -  T_{n,q})$.
	Thus $\gamma_{m1},\gamma_{m2}\notin\varphi_e(T - T_{n,q})$.
	If $\gamma_{m1}\in\phibar(y_{p})$, then $\gamma_{m1}$ is missing twice in
	the ETT $T^*=(T_{n,q},\ e_0,\ y_0,\ e_{1},...,\ e_{p-2},\allowbreak y_{p-2},\ e_p,\ y_p)$ when $p\geq 2$. Here $T^*$ gives a counterexample smaller than $T$. If $p=1$, we must have $q>0$. Otherwise since $T_{n,0}$ is closed for colors in $\phibar_v(T_{n,0})$, we have $\theta\in\phibar(y_{p-1})$, giving a contradiction. We then consider $(T_{n,q},e_1,y_1)$ as a smaller counterlexample. Note that $(T_{n,q},e_1,y_1)$ still satisfies conditions MP, R1 and R2 while dropping $e_0$, as $\varphi(e_1)\in\cup_{\delta_h\in D_{n,q}}\Gamma^{q-1}_h$ and $\varphi(e_1)=\theta\notin\Gamma^{q}$. By Claim \ref{bchange}, $P_{v(\gamma_{m1})}(\delta_m,\gamma_{m1},\varphi)=P_{y_{p-1}}(\delta_m,\gamma_{m1},\varphi)$. Now we consider  $T_{y_p,y_{p-1}}=(T_{n,q},y_0,e_1,...,y_{p-2},e_p,y_p, \allowbreak e_{p-1},y_{p-1})$ obtained from $T$ by switching the order of joining vertices $y_{p}$ and $y_{p-1}$. We can see $T_{y_p,y_{p-1}}$ is also an ETT of $(G,e,\varphi)$ since $\theta \notin\phibar(y_{p-1})$ and $\theta\notin\Gamma^{q}$ and conditions MP, R1, R2 are satisfied.  Applying Claim \ref{bchange} again,  we have $P_{v(\gamma_{m_1})}(\delta_m,\gamma_{m1},\varphi)=P_{y_{p-1}}(\delta_m,\gamma_{m1},\varphi)$, giving a contradiction.
	
	Now we assume $\theta\in\Gamma^{q}$. Without loss of generality we say $\theta=\gamma_{k1}$ for some $k\leq n$. By Claim~\ref{bchange}, $P_{v(\gamma_{m1})}(\delta_m,\gamma_{m1},\varphi)=P_{y_{p-1}}(\delta_m,\gamma_{m1},\varphi)$. If $\delta_k\notin\phibar(y_{p-1})$, $T_{y_p,y_{p-1}}$ also satisfies MP, R1 and R2, where we proceed in the same fashion as the previous case and consider $T_{y_p,y_{p-1}}$. Since $\gamma_{k1}$ can not be both $\gamma_{m1}$ and $\gamma_{m2}$, we assume $\delta_k\in\phibar(y_{p-1})$ and $\gamma_{m2}\neq\gamma_{k1}$. By condition R2,  we have $\gamma_{m2}\notin\varphi_e(T - T_{n,q})$. By
	Claim~\ref{bchange}, $P_{v(\gamma_{m2})}(\delta_m,\gamma_{m2},\varphi)=P_{y_{p-1}}(\delta_m,\gamma_{m2},\varphi)$ and that $P_{y_{p}}(\delta_m,\gamma_{m2},\varphi)$ is different from the path above. Note that $\gamma_{m2}\in\phibar_v(T_{n,q})$, by Claim~\ref{bstablechange}, $T_{n,q}$ is an ETT satisfying conditions MP, R1 and R2 under the $T_{n,q}$-stable coloring $\varphi^*=\varphi/P_{y_{p}}(\delta_m,\gamma_{m2},\varphi)$.   Moreover, $T$ under $\varphi^*$ satisfies conditions MP, R1 and R2 since $\gamma_{m2}\notin\varphi_e(T - T_{n,q})$.  Furthermore, the combination of $\delta_m\in\phibar(y_{p-1})$ and $\theta=\gamma_{k1}$ implies $\delta_m$ is only assigned to connecting edges in $T$.  Therefore, $\gamma_{m2}\notin \varphi^*_e(T -  T_{n,q})$ because $\gamma_{m2}\notin \varphi_e(T -  T_{n,q})$ and no $\delta_m$ edge in $T$ is recolored under $\varphi^*$ . By Claim~\ref{9n},
	$|\phibar^*_v(T_{y_{p-1}})-\varphi^*_e(T_{y_{p-1}} - T_{n,q}))|\geq11+2n$.  Hence there exists $\beta\in\phibar^*(V(T_{y_{p-2}})) - \Gamma^{q}$ such that $\beta\notin\varphi^*_e(T -  T_{n,q})$.
	By Claim~\ref{bchange}, $P_{v(\gamma_{m2})}(\beta,\gamma_{m2},\varphi^*)=P_{v(\beta)}(\beta,\gamma_{m2},\varphi^*)$, and $P_{y_{p}}(\beta,\gamma_{m2},\varphi^*)$ is a different path than above. Let $\varphi^{**}=\varphi^*/P_{y_{p}}(\beta,\gamma_{m2},\varphi^*)$.  Applying Claim~\ref{bstablechange},  we see that  the coloring $\varphi^{**}$ is $T_{n,q}$-stable and $T_{n,q}$ satisfies conditions MP, R1, R2. Moreover, since $\gamma_{m2},\beta\notin\varphi^*_e(T -  T_{n,q})$, $\gamma_{m2},\beta\notin\varphi^{**}_e(T -  T_{n,q})$. Hence $T$ satisfies conditions MP, R1 and R2 under $\varphi^{**}$. Now $\beta\in\phibar^{**}(y_p)$. By
	Claim~\ref{bchange}, $P_{v(\gamma_{k1})}(\beta,\gamma_{k1},\varphi^{**})=P_{v(\beta)}(\beta,\gamma_{k1},\varphi^{**})$, and $P_{y_{p}}(\beta,\gamma_{k1},\varphi^{**})$ is a different path than above. Finally, we let $\varphi^{***}=\varphi^{**}/P_{y_{p}}(\beta,\gamma_{k1},\varphi^{**})$.  Since $\beta\notin\varphi^{**}_e(T -  T_{n,q})$ and by Claim~\ref{bstablechange}, we can see that $T$ satisfies conditions MP, R1 and R2 under $\varphi^{***}$ . However, under $\varphi^{***}$ we have $\varphi^{***}(e_p)=\beta$, $\gamma_{k1}\in\phibar^{***}(y_{p})$ and $v(\beta)\prec_{\ell} y_{p-1}$. Hence
	$(T_{n,q},y_0,e_{1},...,e_{p-2},y_{p-2},e_p,y_p)$ is a counterexample smaller than $T$, giving a contradiction.

	\begin{subsubcase}\label{bcase611b}
		$\theta \in\phibar(y_{p-1})$. In this case $\theta\neq\delta_m$ and $\theta\notin\phibar_v(T_{n,q})$.
	\end{subsubcase}
	Note that $\delta_m$ is only assigned to connecting edges in $T$. By Claim~\ref{bchange}, $P_{v(\gamma_{m1})}(\delta_m,\gamma_{m1},\varphi)=P_{y_{p-1}}(\delta_m,\gamma_{m1},\varphi)$ and $P_{y_{p}}(\delta_m,\gamma_{m1},\varphi)$ is a different path from above. Let $\varphi^*=\varphi/P_{y_{p}}(\delta_m,\gamma_{m1},\varphi)$.  By Claim~\ref{bstablechange},  in $T_{n,q}$-stable coloring $\varphi^*$,  $T_{n,q}$  is an ETT satisfies conditions MP, R1 and R2. Moreover, $\gamma_{m1}\notin\varphi_e(T-T_{n,q})$ and $\delta_m$ is only used by connecting edges in $T$ ensures $T$ satisfies condition R2 under $\varphi^*$. Since $\gamma_{m1}\in\phibar^*_v(T_{n,q})$, by applying Claim~\ref{bchange} again, we  have $P_{v(\gamma_{m1})}(\theta,\gamma_{m1},\varphi^*)=P_{y_{p-1}}(\theta,\gamma_{m1},\varphi^*)$ and $P_{y_{p}}(\theta,\gamma_{m1},\varphi^*)$ is a different path than above. Let $\varphi^{**}=\varphi^*/P_{y_{p}}(\theta,\gamma_{m1},\varphi^*)$.  Now since $\gamma_{m1}\in\phibar^{**}_v(T_{n,q})$, by applying Claim~\ref{bstablechange},  we see that under the $T_{n,q}$-stable coloring $\varphi^{**}$,  $T_{n,q}$ is an ETT  and  satisfies conditions MP, R1 and R2. Moreover, $T$ satisfies conditions MP, R1 and R2 under $\varphi^{**}$ since $\gamma_{m1}\notin\varphi_e(Ty_{p-1}-T_{n,q})$ and $\theta$ may only be used by connecting colors in $T$ if $\theta=\delta_k$ for some $k\leq n$.  Note that under $\varphi^{**}$, 	$\theta\in\phibar^{**}(y_p)\cap\phibar^{**}(y_{p-1})$ and $\varphi^{**}(e_p)=\gamma_{m1}$. If $\theta\in D_{n,q}$, then under $\varphi^{**}$ we have Case~\ref{bcase611a}. So we may assume $\theta\notin D_{n,q}$, which will be handled in case a of Case \ref{bcase612a} below.

	\begin{subcase}\label{bcase612}
		$\alpha\in \phibar(y_p)\cap \phibar(y_{p-1})$ and $\alpha\notin D_{n,q}$.
	\end{subcase}
	\begin{subsubcase}\label{bcase612a}
		$\theta  \notin\phibar(y_{p-1})$.
	\end{subsubcase}
	In this case, $T_{y_p, y_{p-1}} := (T_{n,q}, y_0, e_1, \dots, y_{p-2}, e_p, y_p, e_{p-1}, y_{p-1})$ is also an ETT under $\varphi$ satisfies conditions MP, R1 and R2 except for the case where $\theta\in\Gamma^{q}_m$ for some $m\leq n$ and $\delta_m\in\phibar(y_{p-1})$.
	We first assume there does not exist $1\leq m\leq n$ such that $\theta\in\Gamma^{q}_m$ and $\delta_m\in\phibar(y_{p-1})$.
	By Claim~\ref{9n},  we have $|\phibar_v(Ty_{p-1})-\varphi(E(Ty_{p-1}-T_{n,q}))|\geq 2n+11$.  So there exists  a color $\beta\in\phibar_v(Ty_{p-2}) -  D_{n,q}$ such that $\beta\notin\varphi_e(T - T_{n,q})$. We claim that $\beta\notin\phibar(y_{p})$. Otherwise, $(T_{n,q},y_0,e_{1},...,$ $e_{p-2},$ $y_{p-2},e_p,y_p)$ is a counterexample smaller than $T$, giving a contradiction.   Since $\alpha,\beta\notin D_{n,q}$, by Claim~\ref{bchange} $P_{v(\beta)}(\alpha,\beta,\varphi)=P_{y_{p-1}}(\alpha,\beta,\varphi)$.  Applying  Claim~\ref{bchange} to $T_{y_p, y_{p-1}}$,  we see that $P_{v(\beta)}(\alpha,\beta,\varphi)=P_{y_{p}}(\alpha,\beta,\varphi)$.  So, $P_{v(\beta)}(\alpha,\beta,\varphi)$ has three endvertices $v(\beta)$, $y_{p-1}$ and $y_p$,    a contradiction.  Hence,  we may assume  $\theta=\gamma_{m1}$ and $\delta_{m}\in\phibar(y_{p-1})$ for some $m$ with $ 1\le m\leq n$, which in turn gives  $\gamma_{m2},\alpha\notin\varphi_e(T - T_{n,q})$ and $\delta_m$ is only used by connecting edges in $T$.
	
	By Claim~\ref{bchange}, 	$P_{v(\gamma_{m2})}(\alpha,\gamma_{m2},\varphi)=P_{y_{p-1}}(\alpha,\gamma_{m2},\varphi)$ and $P_{y_{p}}(\alpha,\gamma_{m2},\varphi)$ is a different path from above.  Let $\varphi^*=\varphi/P_{y_{p}}(\alpha,\gamma_{m2},\varphi)$. Since $\gamma_{m2}\in\phibar_v(T_{n,q})$, by applying Claim~\ref{bstablechange},  we see that, under $\varphi^*$,  $T_{n,q}$ is an ETT and satisfies conditions
	MP, R1 and R2. Moreover $T$ satisfies condition MP, R1 and R2 under $\varphi^*$ because $\gamma_{m2},\alpha\notin\varphi_e(T - T_{n,q})$.  If $\delta_m\in\phibar^*(y_{p})$, then with $\delta_m\in\phibar^*(y_p)\cap\phibar^*(y_{p-1})$ and $\varphi^*(e_p)=\gamma_{m1}\notin\phibar^*(y_{p-1})$, where we have Case~\ref{bcase611a}. Hence  $\delta_m\notin\phibar^*(y_{p})$. Now by Claim \ref{bchange}, $P_{v(\gamma_{m_2})}(\delta_m,\gamma_{m2},\varphi^*)=P_{y_{p-1}}(\delta_m,\gamma_{m2},\varphi^*)$, and $P_{y_{p}}(\delta_m,\gamma_{m2},\varphi^*)$ is different from the path above. Let $\varphi^{**}=\varphi^*/P_{y_{p}}(\delta_m,\gamma_{m2},\varphi^*)$.  Since $\gamma_{m2}\in\phibar^*_v(T_{n,q})$, by applying Claim~\ref{bstablechange}, we see that under the $T_{n,q}$-stable coloring $\varphi^{**}$,  $T_{n,q}$ is an ETT satisfies conditions MP, R1 and R2.  Additionally,  $T$ satisfies
	condition MP, R1 and R2  since $\gamma_{m2}\notin\varphi^*_e(T - T_{n,q})$ and $\delta_{m}$ is only assigned to connecting edges in $T$. In $\varphi^{**}$, we have $\delta_m\in\phibar^{**}(y_p)\cap\phibar^{**}(y_{p-1})$, $\gamma_{m1}=\varphi^{**}(e_p)\notin\phibar^{**}(y_{p-1})$, which also leads us back to Case \ref{bcase611a}.
	
	\begin{subsubcase}\label{bcase612b}
		$\theta\in\phibar(y_{p-1})$. In this case,  $\alpha\notin\varphi_e(T - T_{n,q})$.
	\end{subsubcase}
	We first assume  $\theta=\delta_{m}$ for some $m\leq n$.
	By condition R2,  $\gamma_{m1}\notin\varphi_e(T -  T_{n,q})$. By Claim~\ref{bchange}, $P_{v(\gamma_{m1})}(\alpha,\gamma_{m1},\varphi)=P_{y_{p-1}}(\alpha,\gamma_{m1},\varphi)$, and $P_{y_{p}}(\alpha,\gamma_{m1},\varphi)$ is a path different from above. Let $\varphi^*=\varphi/P_{y_{p}}(\alpha,\gamma_{m1},\varphi)$.  Since $\gamma_{m1}\in\phibar_v(T_{n,q})$, we see that
	$\varphi^*$ is $T_{n,q}$-stable and $T_{n,q}$ satisfies conditions MP, R1, R2 under $\varphi^*$ by Claim~\ref{bstablechange}. Moreover $T$ satisfies condition MP, R1 and R2 since $\alpha,\gamma_{m1}\notin\varphi_e(T - T_{n,q})$. Note that $\gamma_{m1}\notin\varphi^*_e(T - T_{n,q})$.  By Claim~\ref{bchange} again, $P_{v(\gamma_{m1})}(\delta_{m1},\gamma_{m1},\varphi^*)=P_{y_{p-1}}(\delta_{m1},\gamma_{m1},\varphi^*)$, and $P_{y_{p}}(\delta_{m1},\gamma_{m1},\varphi^*)$ is a different path from above. Let $\varphi^{**}=\varphi^*/P_{y_{p}}(\delta_{m1},\gamma_{m1},\varphi^*)$. Since $\gamma_{m1}\in\phibar_v(T_{n,q})$, by applying Claim~\ref{bstablechange}, $T_{n,q}$ is an ETT satisfying conditions MP, R1 and R2 under the $T_{n,q}$-stable coloring $\varphi^{**}$. Since $\gamma_{m1}\notin\varphi^*_e(T - T_{n,q})$, $T$ satisfies conditions MP, R1 and R2 under $\varphi^{**}$. Note that under $\varphi^{**}$, we have $\delta_m\in\phibar^{**}(y_p)\cap\phibar^{**}(y_{p-1})$, $\gamma_{m1}=\varphi^{**}(e_p)\notin\phibar^{**}(y_{p-1})$, which is Case~\ref{bcase611a}.  So, we assume $\theta\notin D_{n,q}$.
	
	By Claim \ref{9n}, there exists a color  $\beta\in\phibar_v(T_{y_{p-2}}) -  D_{n,q}$ such that $\beta\notin\varphi_e(T - T_{n,q})$ satisfying either $\beta\notin\Gamma^{q}$ or there exists an $r\in[n]$ with $\beta=\gamma_{r1}\in\Gamma^{q}$ and $\delta_{r}\in\phibar_v(T_{y_{p-2}})$. By Claim~\ref{bchange}, $P_{v(\beta)}(\alpha,\beta,\varphi)=P_{y_{p-1}}(\alpha,\beta,\varphi)$ and $P_{y_{p}}(\alpha,\beta,\varphi)$ is a different path from above. Let $\varphi^*=\varphi/P_{y_{p}}(\alpha,\beta,\varphi)$.  Applying Claim~\ref{bstablechange},  we see that under the $T_{n,q}$-stable coloring $\varphi^*$, $T_{n,q}$ is an ETT satisfying conditions MP, R1 and R2. Moreover $T$ satisfies conditions MP, R1 and R2 under $\varphi^*$ since $\alpha,\beta\notin\varphi_e(T - T_{n,q})$.  In $\varphi^*$, we have $\beta\in\phibar^*(y_p)\cap\phibar^*(v(\beta))$ and $v(\beta)\neq y_{p-1}$. Note that $\theta\notin\varphi^*_e(T_{y_{p-1}} - T_{n,q})$ and $\beta\notin\varphi^*_e(T - T_{n,q})$. Since $\beta,\theta\notin D_{n,q}$, by Claim~\ref{bchange}, $P_{v(\beta)}(\theta,\beta,\varphi^*)=P_{y_{p-1}}(\theta,\beta,\varphi^*)$ and $P_{y_{p}}(\theta,\beta,\varphi^*)$ is different path other than above. Let
	$\varphi^{**}=\varphi/P_{y_{p}}(\theta,\beta,\varphi^*)$.  Applying Claim~\ref{bstablechange} again, we see that under $\varphi^{**}$,  $T_{n,q}$ is an ETT satisfies conditions MP, R1 and R2. Now we check R2 for $T$. Since $\theta\notin\varphi^*_e(T_{y_{p-1}} - T_{n,q})$ and $\beta\notin\varphi^*_e(T - T_{n,q})$, we have R2 being satisfied if $\beta\notin\Gamma^{q}$. For the case when $\beta=\gamma_{r1}\in\Gamma^{q}$, we have $\delta_{r}\in\phibar_v(Ty_{p-2})$ which in turn gives condition R2. Finally by Lemma~\ref{bclear}, $T$ satisfies MP and R1 under $\varphi^{**)}$. However, we have $\theta\in\phibar^{**}(y_p)\cap\phibar^{**}(y_{p-1})$ and $\varphi^{**}(e_p)=\beta\notin\phibar^{**}(y_{p-1})$, which goes back to Case \ref{bcase612a}.

	\begin{subcase}
		$\alpha\in\phibar(y_p)\cap\phibar(v)$ for a vertex  $v\prec_{\ell} y_{p-1}$.
	\end{subcase}

	\begin{CLA}\label{balpha} We may assume
		$\alpha\notin\varphi_e(T - T_{n,q})$ such that  either $\alpha\notin D_{n,q}\cup\Gamma^{q}$ or there exists $k\in[n]$ with
		$\alpha=\delta_k\in\phibar_v(T)$ and $\gamma_{k1},\gamma_{k2}\notin \varphi_e(T - T_{n,q})$.
	\end{CLA}
	
	\proof By Claim~\ref{9n}, we have  $|\phibar_v(T_{y_{p-2}} -D_{n,q}\cup \Gamma^{q}\cup \varphi_e(T_{y_{p-2}} - T_{n,q})| \ge 4$ or there exists index $k$ such that all $\delta_k, \gamma_{k1}$ and $\gamma_{k2} \in \phibar_v(T_{y_{p-2}}) -\varphi_e(T- T_{n,q})$.  The first inequality imples that there exists a color $\beta \in \phibar_v(T_{p-2}) - D_{n,q}\cup \Gamma^{q}\cup \varphi_e(T-T_{n,q})$.  If the second case happens, we take $\beta = \delta_k$.  If $\beta\in\phibar(y_p)$, we are done. Hence we assume $\beta\notin\phibar(y_p)$. Let $P:=P_{y_p}(\alpha, \beta, \varphi)$.  We will show one of the following two statement holds.
	\begin{itemize}
		\item[P:] \label{bStatA} In coloring $\varphi^* = \varphi/P$, $T$ is an ETT and satisfies conditions MP, R1 and R2.
		\item[Q:] \label{bStatB} In $\varphi$, there exists a non-elementary ETT $T'$ with the same ladder as $T$ and $q$ splitters where $T_{n,q}\subseteq T'$ such that conditions MP, R1 and R2 are satisfied, but $p(T') < p(T)$.
	\end{itemize}
	
	Note that statement P gives Claim~\ref{balpha} while statement Q gives a contradiction. Note that $\beta\notin \Gamma^{q}$ by the choice of $\beta$ in Claim~\ref{9n}. We proceed with the proof by considering three cases:  $\alpha \notin \Gamma^{q}$, $\alpha \in \Gamma^{q} - \varphi_e(T-T_{n,q})$ and $\alpha \in \Gamma^{q}\cap \varphi_e(T-T_{n,q})$.
	
	If $V(P)\cap V(T_{y_{p-1}}) = \emptyset$, by Claim~\ref{bstablechange}, we can show that under $\varphi^*=\varphi/P$, $T$ is an ETT and satisfies conditions MP, R1 and R2, so statement P holds. Hence we assume $V(P)\cap V(T_{y_{p-1}})\neq \emptyset$. Along the order of $P$ from $y_p$ , let $u$ be the first vertex in $V(T_{y_{p-1}})$ and $P'$ be the subpath joining $u$ and $y_p$. Let
	\begin{eqnarray*}
		T' & = & T_{y_{p-2}}\cup P' \quad \mbox{ if $u\ne y_{p-1}$, and }\\
		T' & = & T_{y_{p-1}} \cup P' \quad \mbox{ if $u=y_{p-1}$. }
	\end{eqnarray*}
	Note that $e_0\notin T'$ may happen when $q>0$, but it is easy to see that $T'$ is still an ETT with the same ladder as $T$ and $q$ splitters where $T_{n,q}\subseteq T'$.
	{\flushleft \bf Case I: $\alpha\notin \Gamma^{q}$.}	
	Since $\alpha, \beta \in \phibar_v(T_{y_{p-2}})$ and $\alpha\notin \Gamma^{q}$, $T'$ is an ETT satisfying conditions MP, R1 and R2. Hence statement Q holds and gives a contradiction to the minimality of $p(T)$.
	
	
	{\flushleft \bf Case II: $\alpha\in \Gamma^{q}\cap \varphi_e(T-T_{n,q})$.}  Assume $\alpha = \gamma_{m1}$ for some $m \le n$. Since $\varphi(e_p)\neq\alpha$, $\alpha\in\varphi_e(T_{y_{p-1}}-T_{n,q})$. Therefore we must have  $\delta_{m}\in\phibar_v(T_{y_{p-2}})$ by condition R2.  Furthermore, $\beta\in\phibar_v(Ty_{p-2})$. Therefore $T'$ is an ETT and satisfies conditions MP, R1 and R2. Hence statement Q holds.
	
	{\flushleft \bf Case III: $\alpha \in \Gamma^{q} - \varphi_e(T-T_{n,q})$.} Let $\varphi^*=\varphi/P$. By Claim~\ref{bchange}, $P$ is a path different from $P_{v(\alpha)}(\alpha,\beta,\varphi)=P_{v(\beta)}(\alpha,\beta,\varphi)$. Hence $\varphi^*$ is $T_{n,q}$-stable and conditions MP, R1, R2 are satisfied for $T_{n,q}$ by Claim~\ref{bstablechange}. Note that in this case $\varphi(f)=\varphi^*(f)$ for every edge $f$ in $E(T-T_{n,q})$. Therefore $T$ is an ETT satisfying conditions MP, R1 and R2 in coloring $\varphi^*$. Hence statement P holds.  \qed

	Now let $\varphi$ and $\alpha$ be  as the claim above. We then consider two cases.
	\begin{subsubcase}
		$\theta=\varphi(e_p)\notin\phibar(y_{p-1})$.
	\end{subsubcase}
	Let $T'=(T_{n,q},e_0,y_0,e_1,y_1,...,e_{p-2}, y_{p-2},e_p,y_p)$. In this case, $T'$ is an ETT satisfies conditions MP and R1.  Note that $T'$ also satisfies condition R2 with the exception $\theta=\gamma_{m1}$  and $\delta_{m}\in\phibar(y_{p-1})$ for some $m\leq n$, which gives a contradiction to the minimality of $p(T)$. Hence we may assume $\theta=\gamma_{m1}$  and $\delta_{m}\in\phibar(y_{p-1})$ for some $m\leq n$.  By condition R2,  we have  $\gamma_{m1}\notin\varphi_e(T_{y_{p-1}} - T_{n,q})$.
	By Claim~\ref{bchange}, $P_{v} (\alpha, \gamma_{m1}, \varphi) = P_{v(\gamma_{m1})}(\alpha, \gamma_{m1}, \varphi)$ and $P_{y_p}(\alpha, \gamma_{m1}, \varphi)$ is different from the path above.
	Let $\varphi^* =\varphi/P_{y_p}(\alpha, \gamma_{m1}, \varphi)$.
	By Claim~\ref{bstablechange}, $T_{n,q}$ is an ETT satisfying conditions MP, R1 and R2 under the $T_{n,q}$-stable coloring $\varphi^*$ .
	Since $\alpha \notin \varphi_e(T -T_{n,q})$, for every edge $f\in E(T-T_{n,q})$ we have $\varphi^*(f) = \gamma_{m1}$ only if $\varphi(f) = \gamma_{m1}$.  Therefore under $\varphi^*$, $T$ is an ETT satisfying conditions MP, R1 and R2.    Note $\gamma_{m1} \in \phibar^*(y_p)\cap \phibar^*(v(\gamma_{m1}))$.  Since $\delta_m\in \phibar(y_{p-1})$, we have $\delta_m\notin \varphi_e(T-T_{n,q})$.  Let $\varphi^{**} = \varphi^*/P_{y_p}(\delta_m, \gamma_{m1}, \varphi^*)$.  Applying Claim~\ref{bchange} and Claim~\ref{bstablechange}, we can show as before that under $\varphi^{**}$, $T$ is an ETT and satisfies conditions MP, R1, R2, and $\delta_m \in \phibar^{**}(y_p)\cap \phibar^{**}(y_{p-1})$. So,   under $\varphi^{**}$ we go back to  Case \ref{bcase611}.
	
	\begin{subsubcase}
		$\theta=\varphi(e_p)\in\phibar(y_{p-1})$.
	\end{subsubcase}
	We first assume $\theta=\varphi(e_p)=\delta_{m}$ for some $m \leq n$.
	By condition R2,  $\gamma_{m1}\notin\varphi_e(T -  T_{n,q})$. By Claim \ref{bchange}, $P_{v(\gamma_{m1})}(\alpha,\gamma_{m1},\varphi)=P_{v}(\alpha,\gamma_{m1},\varphi)$ and $P_{y_{p}}(\alpha,\gamma_{m1},\varphi)$ is a different path from above. Let $\varphi^*=\varphi/P_{y_{p}}(\alpha,\gamma_{m1},\varphi)$.  By Claim~\ref{bstablechange}, under the $T_{n,q}$-coloring $\varphi^*$, $T_{n,q}$ is an ETT satisfying conditions MP, R1 and R2.
	Since $\alpha, \gamma_{m1}\notin\varphi_e(T -  T_{n,q})$, $T$ is an ETT satisfying MP, R1 and R2 under $\varphi^*$. Now $\gamma_{m1}\in\phibar^*(y_p)$ and $\delta_m, \gamma_{m1} \notin \varphi^*_e(T-T_{n,q})$.  Similarly, by applying Claim~\ref{bchange} and Claim~\ref{bstablechange}, we can show that under the coloring   $\varphi^{**}=\varphi^*/P_{y_{p}}(\delta_m,\gamma_{m1},\varphi^*)$, $T$ is also an ETT satisfying conditions MP, R1 and R2. Now $\delta_m\in\phibar^{**}(y_p)\cap\phibar^{**}(y_{p-1})$, which is dealt in Case \ref{bcase611a}.
	
	We now consider the case $\theta=\varphi(e_p)\notin D_{n,q}$.
	Since $\theta \in \phibar(y_{p-1})$ and $T_{y_{p-1}}$ is elementary, we have $\theta\notin \Gamma^{q}$, so  $\theta\notin D_{n,q}\cup\Gamma^{q}$.
	Suppose $\alpha\neq D_{n,q}$.   Then, $\alpha\neq D_{n,q}\cup\Gamma^{q}$ by Claim~\ref{balpha}. By Claim~\ref{bchange}, $P_{v(\alpha)}(\alpha,\theta,\varphi)=P_{y_{p-1}}(\alpha,\theta,\varphi)$ and $P_{y_p}(\alpha, \theta, \varphi)$ is a different path than the one above. Let $\varphi^*=\varphi/P_{y_{p}}(\alpha,\theta,\varphi)$.  By Claim~\ref{bstablechange}, under $\varphi^*$, $T_{n,q}$ is an ETT satisfying conditions MP, R1 and R2.  Since $\theta,\alpha\notin\varphi_e(T_{y_{p-1}} -  T_{n,q})$ and $\alpha,\theta\notin D_{n,q}\cup\Gamma^{q}$, $T$ is an ETT satisfying conditions MP, R1 and R2 under $\varphi^*$.  Now $\theta\in\phibar^*(y_p)\cap\phibar^*(y_{p-1})$, which is dealt in Case~\ref{bcase611}. Hence we may assume $\alpha=\delta_m\in D_{n,q}$ for some $m\leq n$. By Claim~\ref{balpha}, we have $\gamma_{m1},\gamma_{m2},\delta_m\notin\varphi_e(T - T_{n,q})$. By Claim~\ref{bchange}, $P_{v(\gamma_{m1})}(\alpha,\gamma_{m1},\varphi)=P_{v}(\alpha,\gamma_{m1},\varphi)$ and $P_{y_{p}}(\alpha,\gamma_{m1},\varphi)$ is different from the path above. Let $\varphi^*=\varphi/P_{y_{p}}(\alpha,\gamma_{m1},\varphi)$. Since $\gamma_{m1},\delta_m\notin\varphi_e(T-T_{n,q})$, by Claim~\ref{bstablechange}, $T_{n,q}$ is an ETT satisfying conditions MP, R1 and R2 under $\varphi^*$. Moreover, since $\gamma_{m1},\delta_m\notin\varphi_e(T-T_{n,q})$, $T$ satisfies conditions MP, R1 and R2 under $\varphi^*$. Note  $\gamma_{m1}\notin\varphi^*_e(T - T_{n,q})$. Hence by Claim~\ref{bchange}, $P_{v(\gamma_{m1})}(\theta,\gamma_{m1},\varphi^*)=P_{y_{p-1}}(\theta,\gamma_{m1},\varphi^*)$ and $P_{y_{p}}(\theta,\gamma_{m1},\varphi^*)$ is a different path than above.  Let $\varphi^{**}=\varphi^*/P_{y_{p}}(\theta,\gamma_{m1},\varphi^*)$.
	Again by Claim~\ref{bstablechange},  $T_{n,q}$ is an ETT satisfying conditions MP, R1 and R2 under $\varphi^*$.  Note that from $\varphi^*$ to $\varphi^{**}$, in $E(T -T_{n,q})$, $e_p$ is the only edge changed color from $\theta$ to $\gamma_{m1}$. Since $\delta_m=\alpha\in \varphi^{**}(v)$,   $T$  also an ETT satisfying conditions MP, R1 and R2 under $\varphi^{**}$. Now $\theta\in \phibar^{**}(y_p)\cap \phibar^{**}(y_{p-1})$, which is dealt in Case \ref{bcase612}. This completes Case~\ref{b2}.
	\qed


	In the remainder of the proof, let $I_{\varphi}=\{i\geq0:\phibar(y_p)\cap\phibar(y_{i})\ne \emptyset \}$ and let $j= p(T)$. Clearly $I_{\varphi} = \emptyset$ when $\{ v:\phibar(y_p)\cap \phibar(v) \ne \emptyset\} \subseteq V(T_{n,q})$.  For convention,  we denote $\max(I_{\varphi}) = -1$ when $I_{\varphi} = \emptyset$. By the assumption of $p(T)$,   we have  $j\geq 1$ and $y_{j-1}$ is not incident to $e_j$.

	\begin{case}\label{bcase4}
		$p(T)\leq p-1$.
	\end{case}
	Firstly note that we can assume  $max(I_{\varphi})< p(T)$.
	This is because the case $max(I_{\varphi})\geq p(T)$ is similar to Case \ref{bcase1x} and can be handled in the same fashion:  We first show that $\max(I_{\varphi}) = p-1$  and replace color $\varphi(e_p)$ by $\alpha$ to get a smaller counterexample. Here we omit the details.

	Let $j=p(T)$. Then $j\geq 1$ and $e_j\notin E_G(y_{j-1},y_j)$. Let  $min(I_{\varphi})=i$ if $I_{\varphi}\neq\emptyset$. We let $y_{j-2}$ be the last vertex in $T_{n,q}$ if $j=1$, and in this case $T_{y_{j-2}}=T_{n,q}.$
	
	\begin{CLA}\label{bj-1}
		We may assume there exist $\alpha\in\phibar(y_p)\cap\phibar_v(T_{y_{j-2}})$ such that either $\alpha\notin\Gamma^q$, or	$\alpha=\gamma_{m1}\in\Gamma^{q}$  with $\delta_m\in D_{n,q}$ and $v(\delta_m)\preceq_{\ell} y_{j-2}$.
	\end{CLA}
	
	\proof We first consider the case when $I_{\varphi}\neq\emptyset$. Since we assume $\max(I_{\varphi})<j$, $i\leq j-1$. If $i<j-1$, then $j-2\geq 0$, and we have $\alpha\in\phibar(y_p)\cap\phibar_v(T_{y_{j-2}})$ with $\alpha\notin\Gamma^{q}$ because $\Gamma^q\subseteq\phibar_v(T_{n,q})$ and there is a color in $\phibar(y_{p})\cap \phibar_v(T_{y_{j-2}}-T_{n,q})$.
	Hence we assume $i=j-1$. Thus we have a color $\alpha \in \phibar(y_i)\cap \phibar(y_p)$, and therefore $\alpha \notin \phibar_v(T_{n,q})$ and $\alpha \notin \Gamma^{q}$. By Claim~\ref{9n}, there exists a color $\beta\in\phibar_v(T_{y_{j-2}})$ such that  $\beta\notin\varphi_e(T_{y_{j-1}}-T_{n,q})$ and either $\beta\notin D_{n,q}\cup\Gamma^{q}$ or $\beta=\delta_k\in D_{n,q}$ with $\gamma_{k1},\gamma_{k2},\delta_k\notin\varphi_e(T_{y_{j-1}} -T_{n,q})$.
	
	We now consider the case $\alpha=\delta_{m}\in D_{n,q}$ for some $m\leq n$.
	By condition R2,  $\gamma_{m1}\notin\varphi_e(T_{y_i} - T_{n,q}))$.
	By Claim \ref{bchange},  $P_{v(\gamma_{m1})} (\delta_m, \gamma_{m1}, \varphi)=P_{y_i}(\delta_m, \gamma_{m1}, \varphi)$ and $P_{y_p}(\delta_m, \gamma_{m1}, \varphi)$ is a different path. Let $\varphi^*=$ $\varphi/P_{y_p}(\delta_m,$ $ \gamma_{m_1},\varphi)$.  By Claim~\ref{bstablechange}, under the $T_{n,q}$-stable coloring $\varphi^*$, $T_{n,q}$ satisfies conditions MP, R1 and R2.  Since $\gamma_{m1}\notin \varphi^*_e(T_{y_i}-T_{n,q})$ and $\delta_m$ is only assigned to connecting edges in $T_{y_i}$, $T$ is an ETT satisfying conditions MP, R1 and R2 under $\varphi^*$.  Note that $\gamma_{m1}\notin \varphi^*_e(T_{y_{j-1}} - T_{n,q})$. We have $\gamma_{m1}\notin \varphi^*_e(T_{v(\beta)} - T_{n,q})$. Moreover, $\beta\notin\varphi^*_e(T_{y_{j-1}}-T_{n,q})$. By Claim~\ref{bchange}, $P_{v(\gamma_{m1})} (\beta, \gamma_{m1}, \varphi^*) = P_{v(\beta)} (\beta, \gamma_{m1}, \varphi^*)$ and $P_{y_p}(\beta, \gamma_{m1}, \varphi^*)$ is a different path.  Let  $\varphi^{**}=\varphi^*/P_{y_p}(\gamma_{m_1},\beta,\varphi^*)$.
	By Claim~\ref{bstablechange}, $T_{n,q}$ satisfies conditions MP, R1 and R2 under the $T_{n,q}$-stable coloring $\varphi^{**}$. Since $\beta, \gamma_{m1} \notin \varphi^*_e(T_{y_{j-1}} -T_{n,q})$, we have $\beta, \gamma_{m1} \notin \varphi^{**}_e(T_{y_{j-1}} -T_{n,q})$. Thus, under $\varphi^{**}$, $T$ is an ETT satisfying conditions MP, R1 and R2.   Now $\beta\in\phibar^{**}(y_p)\cap \phibar^{**}(v(\beta))$,  So Claim~\ref{bj-1} holds.

	We now consider the case $\alpha \notin D_{n,q}$. Recall that $\alpha\notin\Gamma^{q}$. Since  $\alpha \in \phibar(y_{j-1})\cap \phibar(y_p)$, $\alpha\notin\varphi_e(T_{y_{j-1}}-T_{n,q})$. We first assume $\beta\notin D_{n,q}\cup\Gamma^{q}$. Since $\alpha, \beta \notin \varphi_e(T_{y_{j-1}} - T_{n,q})$,  by Claim~\ref{bchange}, $P_{v(\beta)} (\alpha, \beta, \varphi) =
	P_{y_i}(\alpha, \beta, \varphi)$ and $P_{y_p}(\alpha, \beta, \varphi)$ is a different path.  By Claim~\ref{bstablechange},   under $\varphi^*=\varphi/P_{y_p}(\alpha,\beta,\varphi)$, $T_{n,q}$ satisfies conditions MP, R1 and R2. For any edge $f\in E(T_{y_{j-1}})$, $\varphi(f) = \varphi^*(f)$ since $\alpha,\beta\notin\varphi_e(T_{y_{j-1}}-T_{n,q})$. Since $\alpha,\beta\in\phibar^*_v(T_{y_{j-1}})$, $T$ is still an ETT satisfying conditions MP and R1 under $\varphi^*$ by Lemma~\ref{bclear}. Since both $\alpha, \beta\notin D_{n,q}\cup \Gamma^{q}$, $T$ also satisfies condition R2 under $\varphi^*$.  It is seen that, under $\varphi^*$, Claim~\ref{bj-1} holds.
	
	We now assume $\beta=\delta_m\in D_{n,q}$ for some $m\leq n$.  By our choice of $\beta$, we have
	$\delta_m, \gamma_{m1}, \gamma_{m2}\notin \varphi_e(T_{y_i} - T_{n,q})$.    By Claim~\ref{bchange},
	$P_{v(\gamma_{m1})} (\alpha, \gamma_{m1}, \varphi) = P_{y_i} (\alpha, \gamma_{m1}, \varphi)$ and $P_{y_p} (\alpha, \gamma_{m1}, \varphi)$ a different path. Let $\varphi^*=\varphi/P_{y_p}(\alpha, \gamma_{m1},\varphi)$. By Claim~\ref{bstablechange}, $\varphi^*$ is $T_{n,q}$-stable and $T_{n,q}$ satisfies
	conditions MP, R1 and R2 under $\varphi^*$.  Since $\alpha, \gamma_{m1} \notin \varphi_e(T_{y_i} - T_{n,q})$, we have
	$\alpha, \gamma_{m1} \notin \varphi^*_e(T_{y_i} - T_{n,q})$.  So, as an ETT under $\varphi^*$, $T_{n,q}$ can be extended to $T$ and condition R2 holds.  
	Now we have as claimed under $\varphi^{*}$ because $\beta=\delta_m\in\phibar^*_v(T_{y_{j-2}})$.
	
	We then consider the case $I_{\varphi}=\emptyset$. If $\alpha\notin\Gamma^{q}$, we are done. Hence we assume $\alpha=\gamma_{m1}\in\Gamma^{q}$ for some $m\leq n$. We first assume that $\delta_m\notin\phibar_v(T_{y_{p-1}})$. Then $\gamma_{m1}\notin\varphi_e(T-T_{n,q})$. By Claim~\ref{9n}, there exists a color $\beta\in\phibar_v(T_{y_{p-2}})$ such that  $\beta\notin\varphi_e(T_{y_{p}} - T_{n,q})$ and either $\beta\notin D_{n,q}\cup\Gamma^{q}$ or $\beta=\delta_k\in D_{n,q}$ with $\gamma_{k1},\gamma_{k2},\delta_{k}\notin\varphi_e(T_{y_{p}} -T_{n,q})$. By Claim~\ref{bchange},
	$P_{v(\gamma_{m1})} (\beta, \gamma_{m1}, \varphi) = P_{v(\beta)} (\beta, \gamma_{m1}, \varphi)$ and $P_{y_p} (\beta, \gamma_{m1}, \varphi)$ is a different path. By Claim~\ref{bstablechange}, we have that $T_{n,q}$ is an ETT satisfying conditions MP, R1, R2 under the $T_{n,q}$-stable coloring $\varphi^*=\varphi/P_{y_p}(\beta, \gamma_{m1},\varphi)$. Since $\beta, \gamma_{m1} \notin \varphi_e(T_{y_p} - T_{n,q})$, we have
	$\beta, \gamma_{m1} \notin \varphi^*_e(T_{y_p} - T_{n,q})$.  So, as an ETT under $\varphi^*$, $T_{n,q}$ can be extended to an ETT $T$ which satisfies conditions MP, R1 and R2. Note that under $\varphi^*$ we have Claim~\ref{bj-1} if $v(\beta)\prec_{\ell} y_{j-2}$. If $y_j \prec_{\ell} v(\beta)$, we have $max(I_{\varphi})\geq p(T)$. If $v(\beta)=y_{j-1}$, we have the case $I_{\varphi}\neq\emptyset$, where we can proceed as before.
	
	We now assume that $\delta_m\in\phibar_v(T_{y_{p-1}})$. Since $\delta_m\in D_{n,q}$,  $\delta_m\notin\phibar_v(T_{n,q})$. Without loss of generality, we assume that $\delta_{m}\in\phibar(y_k)$ for some $k\leq p-1$. If $k< j-1$, we have Claim~\ref{bj-1}, hence we assume $k\geq j-1$. By condition R2, $\gamma_m\notin\varphi_e(T_{y_k}-T_{n,q})$.  By Claim~\ref{bchange},
	$P_{v(\gamma_{m1})} (\delta_m, \gamma_{m1}, \varphi) = P_{y_k} (\delta_m, \gamma_{m1}, \varphi)$ and $P_{y_p} (\delta_m, \gamma_{m1}, \varphi)$ a different path. By Claim~\ref{bstablechange},
	under the $T_{n,q}$-stable coloring $\varphi^*:=\varphi/P_{y_p}(\delta_m, \gamma_{m1},\varphi)$,  $T_{n,q}$ satisfies
	conditions MP, R1 and R2. Since $\gamma_{m1},\delta_m\notin\varphi_e(T_{y_{k}}-T_{n,q})$, $T$ satisfies conditions MP, R1, and R2 under $\varphi^*$ because the edges of $T$ which are colored different under $\varphi^*$ and $\varphi$ is in $T_{y_{p}}-T_{y_{k}}$, and they are colored by $\gamma_{m1}$ or $\delta_m$ in both colorings $\varphi$ and $\varphi^*$. If $k>j-1$, we have $max(I_{\varphi})\geq p(T)$.  If $k=j-1$, we have the case $I_{\varphi}\neq\emptyset$, where we can proceed as before.
	\qed

	We assume $\varphi$ satisfies Claim~\ref{bj-1}.
	By Claim~\ref{9n},  there exists  a color $\beta\in\phibar_v(T_{y_{j-2}})$ with
	$\beta\notin\varphi_e(T_{y_{j}} -  T_{n,q})$ such that  either $\beta\notin D_{n,q}\cup\Gamma^{q}$ or $\beta=\delta_k\in D_{n,q}$  with $\delta_k,\gamma_{k1},\gamma_{k2}\notin\varphi_e(T_{y_{j}}  -  T_{n,q})$ for some $k\leq n$. Now we consider the path $P:=P_{y_{p}}(\alpha,\beta,\varphi)$. First we consider the case $V(P)\cap V(T_{y_{j-1}})\neq\emptyset$. Along the order of $P$ from $y_p$ , let $u$ be the first vertex in $V(T_{y_{j-1}})$ and $P'$ be the subpath joining $u$ and $y_p$. Let
	\begin{eqnarray*}
		T' & = & T_{y_{j-2}}\cup P' \quad \mbox{ if $u\ne y_{j-1}$, and }\\
		T' & = & T_{y_{j-1}} \cup P' \quad \mbox{ if $u=y_{j-1}$. }
	\end{eqnarray*}
	Again note that $e_0\notin T'$ may happen when $q>0$, but it is easy to see that $T'$ is still an ETT with the same ladder as $T$ and $q$ splitters where $T_{n,q}\subseteq T'$.
	{\flushleft \bf Case I: $\alpha\notin \Gamma^{q}$.}	
	Since $\alpha, \beta \in \phibar_v(T_{y_{j-2}})$, $T'$ is an ETT satisfying conditions MP, R1 and R2, giving a contradiction to the minimality of $p(T)$.
	
	
	{\flushleft \bf Case II: $\alpha\in \Gamma^{q}$.} Then by Claim~\ref{bj-1}, $\alpha=\gamma_{m1}\in\Gamma^{q}$ for some $m\leq n$ and $v(\delta_m)\prec_{\ell} y_{j-2}$. Then $\delta_{m}\in\phibar_v(T_{y_{j-2}})$.  Furthermore, $\beta\in\phibar_v(T_{y_{j-2}})$. Therefore $T'$ is an ETT and satisfies conditions MP, R1 and R2, giving a contradiction to the minimality of $p(T)$. 	
	

	Therefore we have $V(P)\cap V(T_{y_{j-1}})=\emptyset$. Let $\varphi^*=\varphi/P$. Then $\varphi^*$ is $T_{y_{j-1}}$-stable and $T_{y_{j-1}}$ satisfies conditions MP, R1 and R2 by Lemma~\ref{bLEM:Stable}. If $\alpha\notin\Gamma^{q}$, $T$ satisfies conditions MP, R1 and R2 under $\varphi^*$ since $\alpha,\beta\in\phibar_v(T_{y_{j-2}})$ and $\beta\notin\Gamma^q$. If $\alpha\in\Gamma^{q}$, by Claim~\ref{bj-1},  $\alpha=\gamma_{m1}\in\Gamma^{q}$ for some $m\leq n$ and $v(\delta_m)\prec_{\ell} y_{j-2}$. Therefore $T$ satisfies conditions MP, R1 and R2  since $\beta,\delta_m\in\phibar_v(T_{y_{j-2}})$ and $\beta\notin\Gamma^q$.  Note that $\beta\notin\varphi^*_e(T_{y_j}-T_{n,q})$ and $\beta\in\phibar^*(y_p)\cap\phibar^*(v(\beta))$, where $v(\beta)\prec_l y_{j-2}$. Denote $v=v(\beta)$ for convenience.
	Let $\gamma\in\phibar(y_j)$. Then $\gamma\notin\Gamma^{q}$ and $\gamma\notin\varphi^*(T_{y_j}-T_{n,q})$.
	We then denote $\varphi^*=\varphi$ and consider the following two cases.
	
	\begin{subcase}$\gamma\notin D_{n,q}$.
	\end{subcase}
	\begin{subsubcase}
		$\beta\notin D_{n,q}$.
	\end{subsubcase}
	
	By Claim \ref{bchange}, $P_{v(\beta)} (\beta, \gamma, \varphi) = P_{y_j} (\beta, \gamma, \varphi)$ and $P_{y_p} (\beta, \gamma, \varphi)$ is a different path. Let $\varphi^*=\varphi/P_{y_p}(\beta,\gamma,\varphi)$. Then by Claim~\ref{bstablechange}, $\varphi^*$ is $T_{n,q}$-stable and $T_{n,q}$ satisfies conditions MP, R1, R2 under $\varphi^*$. Since $\gamma,\beta\notin\varphi_e(T_{y_j}-T_{n,q})$ and $v\prec_l y_j$, and moreover $\gamma,\beta\notin\Gamma^{q}$, $T$ satisfies conditions MP, R1, and R2 under $\varphi^*$. Now $\gamma\in\phibar^*(y_p)\cap\phibar^*(y_j)$, where we have $max(I_{\varphi})\geq p(T)$.
	\begin{subsubcase}$\beta=\delta_m\in D_{n,q}$ for some $m \leq n$.
	\end{subsubcase}
	In this case $\gamma_{m1},\gamma_{m2}\notin\varphi((T_{y_{j}}- T_{n,q}))$ by our choice on $\beta$. By Claim \ref{bchange}, $P_{v(\beta)} (\beta, \gamma_{m1}, \varphi) = P_{v(\gamma_{m1})} (\beta, \gamma_{m1}, \varphi)$ and $P_{y_p} (\beta, \gamma_{m1}, \varphi)$ a different path. Let $\varphi^*=\varphi/P_{y_p}(\gamma_{m1},\beta,\varphi)$. By Claim~\ref{bstablechange}, $T_{n,q}$ is an ETT satisfying conditions MP, R1 and R2 under the $T_{n,q}$-stable coloring $\varphi^*$. Moreover, $T$ satisfies conditions MP, R1 and R2 under $\varphi^*$ since $\delta_m=\beta\in\phibar_v(T_{y_{j-2}})$ and $\beta,\gamma_{m1}\notin\varphi_e(T_{y_{j}}-T_{n,q})$.
	Similarly by Claim \ref{bchange} again, $P_{v(\beta)} (\gamma, \gamma_{m1}, \varphi^*) = P_{v(\gamma_{m1})} (\gamma, \gamma_{m1}, \varphi^*)$ and $P_{y_p} (\gamma, \gamma_{m1}, \varphi^*)$ a different path. Note that we have $\gamma,\gamma_{m1}\notin\varphi^*_e(T_{y_{j}}-T_{n,q})$. Let $\varphi^{**}=\varphi^*/P_{y_p}(\gamma_{m1},\gamma,\varphi^*)$. By Claim~\ref{bstablechange}, $T_{n,q}$ satisfies conditions MP, R1 and R2 under the $T_{n,q}$-stable coloring $\varphi^{**}$. Since $\delta_m\in\phibar^*_v(T_{y_{j-2}})$ and $\gamma,\gamma_{m1}\notin\varphi^*_e(T_{y_{j}}-T_{n,q})$, $T$ satisfies conditions MP, R1 and R2 under $\varphi^{**}$.
	However, we have $\gamma\in\phibar^{**}(y_p)\cap\phibar^{**}(y_j)$, where we have $max(I_{\varphi})\geq p(T)$.

	\begin{subcase} $\gamma=\delta_m\in D_{n,q}$ for some $m \leq n$.
	\end{subcase}
	By condition R2, $\gamma_{m1},\gamma_{m2}\notin\varphi_e(T_{y_{j}}-T_{n,q})$. Recall that $\beta\notin\varphi_e(T_{y_{j}}-T_{n,q})$. By Claim \ref{bchange}, $P_{v(\beta)}(\beta, \gamma_{m1}, \varphi) = P_{v(\gamma_{m1})} (\beta, \gamma_{m1}, \varphi)$ and $P_{y_p} (\beta, \gamma_{m1}, \varphi)$ is a different path. Let $\varphi^*=\varphi/P_{y_p}(\gamma_{m1},\beta,\varphi)$. By Claim~\ref{bstablechange}, $T_{n,q}$ is an ETT satisfying conditions MP, R1 and R2 under the $T_{n,q}$-stable coloring $\varphi^*$. Moreover, $T$ satisfies conditions MP, R1 and R2 under $\varphi^*$ since $\delta_m=\gamma\in\phibar_v(T_{y_{j}})$ and $\beta,\gamma_{m1}\notin\varphi_e(T_{y_{j}}-T_{n,q})$. Note that $\beta,\gamma_{m1}\notin\varphi^*_e(T_{y_{j}}-T_{n,q})$.
	Similarly by Claim \ref{bchange}, $P_{v(\beta)} (\gamma, \gamma_{m1}, \varphi^*) = P_{v(\gamma_{m1})} (\gamma, \gamma_{m1}, \varphi^*)$ and $P_{y_p} (\gamma, \gamma_{m1}, \varphi^*)$ is a different path.  Let $\varphi^{**}=\varphi^*/P_{y_p}(\gamma_{m1},\gamma,\varphi^*)$. By Claim~\ref{bstablechange}, $T_{n,q}$ satisfies conditions MP, R1 and R2 under $\varphi^{**}$. Since $\delta_m\in\phibar^*_v(T_{y_{j}})$ and $\gamma,\gamma_{m1}\notin\varphi^*_e(T_{y_{j}}-T_{n,q})$, $T$ satisfies conditions MP, R1 and R2 under $\varphi^{**}$.
	Now we have
	$\delta_m\in\phibar^{**}(y_p)\cap\phibar^{**}(y_j)$, where we have $max(I_{\varphi})\geq p(T)$.

	This completes the proof of Case \ref{bcase4}. Now for all cases we arrive at a contradiction, which proved statement A.
\end{proof}

\begin{REM}\label{brm2}
	We can see from the proof of Theorem~\ref{bmain}, conditions MP and R1 are only used to show $T=T_k\cup\{f_k,b(f_k)\}$ is elementary if $T_k$ satisfies MP and R1 where $f_k$ is a connecting edge for each $0<k\leq n$. All the techniques we used during the proof are about reducing the number of vertices to obtain a contradiction with $T_k\cup\{f_k,b(f_k)\}$ being elementary. Therefore, if we can figure out new ways of adding a vertex to a closed ETT $T_k$ and proving the resulting ETT is elementary without R1, we may have a chance to tackle the conjecture.
	
\end{REM}

\bibliographystyle{plain}
\bibliography{Gold2016Sep22}






\end{document}